\documentclass[11pt,leqno]{amsart}

\usepackage[english]{babel}
\usepackage{url}
\usepackage[latin1]{inputenc}
\usepackage[all,2cell,emtex]{xy}
\usepackage{amssymb,amsmath,bbm,xr,color,stmaryrd,xfrac}
\usepackage[mathscr]{euscript}
\usepackage{mathrsfs}
\usepackage[colorlinks=true,pagebackref=true]{hyperref} 

\newcommand{\Sch}{\mathscr Sch}
\newcommand{\Reg}{\mathscr Reg}
\newcommand{\Sm}{\mathscr Sm}

\DeclareMathOperator{\SH}{SH}
\DeclareMathOperator{\Htp}{\mathscr H}
\DeclareMathOperator{\Htpp}{\mathscr H_\bullet}
\DeclareMathOperator{\V}{\mathscr V}

\DeclareMathOperator{\rk}{rk}

\newcommand{\SSp}{\mathbb S}
\newcommand{\E}{\mathbb E}

\newcommand{\un}{\mathbbm 1}

\newcommand{\KGL}{\mathbf{KGL}}
\newcommand{\KSp}{\mathbf{KSp}}
\newcommand{\KQ}{\mathbf{GW}}
\newcommand{\KW}{\mathbf{W}}
\newcommand{\HMW}{\mathbf{H}_{\mathrm{MW}}}
\newcommand{\HM}{\mathbf{H}_{\mathrm M}}
\DeclareMathOperator{\HH}{\mathbf H}
\newcommand{\uKMW}{\underline{\mathrm{K}}^{\mathrm{MW}}}
\newcommand{\HW}{\mathrm{W}}
\newcommand{\piaone}{\boldsymbol{\pi}^{{\mathbb A}^1}}

\newcommand{\GL}{\mathrm{GL}}
\newcommand{\SL}{\mathrm{SL}}
\newcommand{\Sp}{\mathrm{Sp}}

\newcommand{\thom}{\mathfrak t}

\DeclareMathOperator{\ch}{ch}
\DeclareMathOperator{\bo}{bo}

\DeclareMathOperator{\Hom}{Hom}
\DeclareMathOperator{\Aut}{Aut}
\DeclareMathOperator{\End}{End}

\DeclareMathOperator{\CH}{CH}
\newcommand{\wCH}{\widetilde{\mathrm{CH}}{}}
\DeclareMathOperator{\GW}{GW} 
\DeclareMathOperator{\W}{W} 
\DeclareMathOperator{\uW}{\underline W} 
\DeclareMathOperator{\Spec}{Spec}
\DeclareMathOperator{\Sym}{Sym}
\newcommand{\KM}{\mathrm{K}^{\mathrm{M}}}
\newcommand{\KMW}{\mathrm{K}^{\mathrm{MW}}}

\newcommand{\dtw}[1]{\langle #1 \rangle}
\DeclareMathOperator{\Th}{Th}

\DeclareMathOperator{\K}{K}

\newcommand{\derR}{\mathbf{R}}

\newcommand{\NN} {\mathbb N}
\newcommand{\ZZ} {\mathbb Z}
\newcommand{\QQ} {\mathbb Q}
\newcommand{\RR} {\mathbb R}

\newcommand{\ZZp} {\mathbb Z[\sfrac 1 2]}
\renewcommand{\AA} {\mathbb A}
\newcommand{\PP} {\mathbb P}
\newcommand{\HP}{\mathrm{H}\mathbb{P}}
\newcommand{\GG} {\mathbb{G}_m}

\newcommand{\zar}{\mathrm{Zar}}

\newcommand{\cU}{\mathfrak U}
\newcommand{\cV}{\mathfrak V}
\newcommand{\cW}{\mathfrak W}
\newcommand{\cH}{\mathfrak H} 

\title{The Borel character}

\author{Fr\'ed\'eric D\'eglise}
\address{Institut de Math \'ematiques de Bourgogne, UMR 5584, CNRS, Universit\'e Bourgogne Franche-Comt\'e, F-21000 Dijon, France}
\email{frederic.deglise@ens-lyon.fr}
\urladdr{http://perso.ens-lyon.fr/frederic.deglise/}
 
\author{Jean Fasel}
\address{Institut Fourier - UMR 5582, Universit\'e Grenoble-Alpes, CS 40700, 38058 Grenoble Cedex 9, France}
\email{Jean.Fasel@univ-grenoble-alpes.fr}
\urladdr{https://www-fourier.univ-grenoble-alpes.fr/~faselj/}

\date{\today}

\newtheorem{thm}{Theorem}[subsection]
\newtheorem{prop}[thm]{Proposition}
\newtheorem{lm}[thm]{Lemma}
\newtheorem{cor}[thm]{Corollary}
\newtheorem{thmi}{Theorem}

\theoremstyle{remark} 
\newtheorem{rem}[thm]{Remark}

\newtheorem{ex}[thm]{Example}

\theoremstyle{definition} 
\newtheorem{df}[thm]{Definition}
\newtheorem{num}[thm]{}

\numberwithin{equation}{thm}

\newtheorem{thm*}{Theorem}

\begin{document}

\begin{abstract}
The main purpose of this article is to define a \emph{quadratic} analog of the Chern character,
 the so-called Borel character, which identifies rational higher Grothendieck-Witt groups 
 with a sum of rational MW-motivic cohomologies and rational motivic cohomologies.
We also discuss the notion of ternary laws due to Walter,
 a quadratic analog of formal group laws, and compute what we call the additive ternary laws,
 associated with MW-motivic cohomology. Finally, we provide an application of the Borel character by showing that the Milnor-Witt $K$-theory of a field $F$ embeds into suitable higher Grothendieck-Witt groups of $F$ modulo explicit torsion.
\end{abstract}

\maketitle

\setcounter{tocdepth}{1}
\tableofcontents

\section{Introduction}

\subsection*{Quadratic forms, topology and $\AA^1$-homotopy}

The K-theoretic study of quadratic forms has its origin in Bass' 1965-66 lectures 
 at the Tata institute (see \cite{Bass}),
 building on Wall's invariant of a quadratic form;
 the zeroth and first groups of what was later called \emph{hermitian K-theory} appeared there.
 It was quickly linked with the so-called "surgery problem"
 by Wall in \cite{Wall}, where the L-group variant was introduced. 
 The term hermitian K-theory, for Bass' initial groups and their higher version,
 was introduced by Karoubi and Villamayor in their CRAS paper \cite{KarVill} dated from 1971.
 One year earlier, in his famous paper \cite{Milnor}, Milnor provided the first links between K-theory defined via symbols
 (\emph{i.e.} Milnor K-theory), Galois cohomology and Witt groups, sowing the seed of
 what has become \emph{motivic homotopy theory}.
 The theory of Morel and Voevodsky --- introduced in 2000 and originally called $\AA^1$-homotopy theory ---
 has soon revealed its connections with quadratic forms under the leadership of Morel.
 Indeed, the analysis of Voevodsky's proof of Milnor conjecture led Morel to
 the main computation of $\AA^1$-homotopy theory up to now: the identification of the $0$-th stable homotopy group of the sphere spectrum
 with the Grothendieck-Witt ring of the base field $k$,
 and further of the $\ZZ$-graded $0$-th stable homotopy group\footnote{that is, with respect to the $\GG$-grading: recall
 the formula $\pi_0(S^0)_n:=[S^0,\GG^n]^{stable}$.} with the so-called Milnor-Witt K-theory of $k$.
 See \cite[Cor. 1.25]{MorLNM}.

Before going further into the quadratic part of motivic homotopy theory,
 let us go back in history for a moment; to one of the sources of Beilinson's program,
 the \emph{Chern character}. From the initial point of view, say over a smooth $k$-variety $X$
 (more generally, a regular scheme $X$),
 this is a rational ring isomorphism between the Grothendieck group of vector bundles over $X$
 and the Chow ring of $X$; an incarnation of Serre's intersection Tor-formula.
 Grothendieck initial's breakthrough has motivated a thorough line of research in algebraic topology,
 which tries to classify spectra (\emph{i.e} representable cohomology theories)
 in terms of their characteristic classes.
 After Quillen and Adams, to any (complex) oriented spectra is associated a \emph{formal group law} $F(x,y)$
 which expresses the behavior of its associated first Chern classes with respect to tensor products:
For line bundles $L$ and $L'$, we have
\begin{equation}\label{eq:FGL}\tag{FGL}
c_1(L \otimes L')=F(c_1(L),c_1(L')).
\end{equation}
This gives the following classical results:
\begin{itemize}
\item singular cohomology is the universal such theory with additive FGL;
\item complex K-theory is the universal one with multiplicative FGL;
\item cobordism is the universal one with the universal FGL.
\end{itemize}
The topological Chern character is then interpreted as the unique rational
 morphism of oriented spectra from the multiplicative one to the additive one.\footnote{Which,
 in terms of formal group laws corresponds to the exponential power series.}
 Further, rationally, all cohomologies are "ordinary": a direct sum of copies of
 singular cohomologies, in particular oriented.

The notion of oriented cohomology theory was naturally extended to $\AA^1$-homotopy theory,
 giving the following table of analogies:\footnote{The universality of motivic cohomology
 is obtained over any base in \cite{CD3}; over a singular base, Quillen algebraic K-theory has to be replaced
 by homotopy invariant K-theory \cite{CisKth}; the universality of the FGL associated with algebraic cobordism
 is due to Levine and Morel \cite{LM}.}
\begin{center}
\begin{tabular}{c|c|c}
topology & geometry & FGL \\
\hline
singular cohomology & motivic cohomology & additive \\
complex K-theory & algebraic K-theory & multiplicative \\
complex cobordism & algebraic cobordism & universal
\end{tabular}
\end{center}
In the motivic context, the Chern character was constructed by Riou (cf. \cite{Riou}),
 extending the initial work of Gillet and Soul\'e (cf. \cite{SouOp}).
 Despite this appealing analogy, in rational stable $\AA^1$-homotopy not all
 cohomologies reduce to motivic cohomology. For example, Chow-Witt groups
 are not oriented in general. Back to our starting point,
 Hermitian K-theory, though representable (over regular bases) is also non orientable,
 even with rational coefficients.

\subsection*{Panin-Walter weak orientations}

Motivated by these examples, Panin and Walter introduced in a series of fundamental papers weaker notion of orientation (\cite{PW1, PW2}).
 Recall that an orientation on a ring spectrum $\E$ in the stable homotopy category $\SH(S)$
 can be expressed as the data for each vector bundle $V/X$ over a smooth $S$-scheme $X$ of rank $n$
 of an isomorphism, called the Thom isomorphism:
$$
\thom(V):\E^{*,*}(\Th(V)) \xleftarrow{\ \sim\ } \E^{*-2n,*-n}(X)
$$
where $\E^{**}$ is the associated (bigraded) cohomology, and $\Th(E)=E/E^\times$ is the Thom space of $E$.
 The idea of Panin and Walter is to ask for the existence of Thom isomorphisms only for a restricted
 class of vector bundles; namely, those corresponding to $G$-torsors for linear algebraic groups
 such as $\SL_n$ or $\Sp_{2r}$ (in which case, $n=2r$ is even). This gives rise to the notions
 of $\SL$-orientation and $\Sp$-orientation
 (the later is weaker; see Definition~\ref{df:worientations} for details).
 The extraordinary cohomology theories mentioned above fulfill this new axiomatization.
 Moreover, they are organized in mirror of the classically oriented cohomologies,
 as described in the following table:
\begin{center}
\begin{tabular}{c|c}
$\GL$-oriented & $\SL$-oriented \\
\hline
motivic cohomology & MW-motivic cohomology \\
algebraic K-theory & hermitian K-theory/higher GW-theory \\
algebraic cobordism & special linear cobordism.
\end{tabular}
\end{center}
The term \emph{MW-motivic cohomology} stands for Milnor-Witt cohomology.
 It is to Chow-Witt groups what motivic cohomology groups are to usual Chow groups.
 We will denote by $\HMW R_S$ the ring spectrum which represent $R$-linear 
 MW-motivic cohomology over a scheme $S$, and refer the reader to our conventions
 at the end of this introduction for more precision.
 Note also that higher GW-theory is a shortcut for \emph{higher Grothendieck-Witt groups}.
 This terminology is used in \cite{ST}. Note that higher GW-theory agree with
 K-theory of symplectic (resp. symmetric) bundles in bidegree $(4n+2,2n+1)$ 
 (resp. $(4n,2n)$).\footnote{At the moment, it is only defined for $\ZZp$-schemes
 so that all our results concerning higher GW-theory will be stated under this
 assumption.}
 
Let us focus on the notion of $\Sp$-orientation in this introduction.
 Recall that an $\Sp$-torsor corresponds to a vector bundle $V$ equipped with a non-degenerate symplectic form $\psi$,
 that we call \emph{symplectic bundles}. Panin and Walter have associated to any symplectic bundle
 a quaternionic projective bundle $\HP(V,\psi)$, equipped with a canonical rank 2 symplectic bundle
 (see Paragraph~\ref{num:sp_pb_thm} for details) --- below, $\HP^\infty$ denotes the infinite dimensional
quaternionic projective space.
 Given these considerations, the beauty of $\Sp$-orientations, say on a ring spectrum $\E$,
 is to be perfectly analogous to classical orientations, as summarized in the following
 table
\begin{center}
\begin{tabular}{c|c|c}
 & $\GL$ & $\Sp$ \\
\hline
orientation as a class & $c \in \E^{2,1}(\PP^\infty)$ & $b \in \E^{4,2}(\HP^\infty)$ \\
Thom classes of & vect. bundles & symplectic vect. bundles \\
proj. bundle formula & proj. bundle & quaternionic proj. bundle \\
characteristic classes & Chern classes $c_i$ & Borel classes $b_i$. \\
\end{tabular}
\end{center}
We refer the reader to Section~\ref{sec:worientations} for details on the notions and definitions 
appearing on the right-hand side.
 Given this perfect table, one is naturally led to wonder what is the analog of formal group laws in
 in the symplectically oriented case. The problem with the formula \eqref{eq:FGL} in the symplectic case
 is that a tensor product of two symplectic bundles is not a symplectic bundle: indeed,
 the tensor product of two symplectic forms is not symplectic, but symmetric.
 However, a triple product of symplectic forms is again symplectic.

\subsection*{Walter's ternary laws}
This led to the notion of \emph{ternary law} which we introduce in this paper
 following an unpublished work of Walter.
 The idea is formally identical to that of formal group laws,
 except that we work with $P=\HP^\infty_S$ over some base $S$
 and triple products.
 Let us give the formula for the comfort of the reader (see Definition~\ref{df:ternary}).
 On $P^3$, we get three tautological symplectic bundles $\cU_1$, $\cU_2$, $\cU_3$.
 Then $\cU_1 \otimes \cU_2 \otimes \cU_3$ has rank $8$, and admits $4$ non-trivial
 Borel classes. In particular, one associates to an $\Sp$-oriented ring spectrum $\E$
 over $S$ four power series in three variables $x,y,z$, say:
$$
F_l(x,y,z)=\sum_{i,j,k \geq 0} a^l_{ijk}.x^iy^jz^k, l=1,2,3,4
$$
defined by the formula in $\E^{**}(S)[[b_1(\cU_1),b_1(\cU_2),b_1(\cU_3)]]$:
\begin{align*}
F_l\big(b_1(\cU_1),b_1(\cU_2),b_1(\cU_3)\big):=b_l(\cU_1 \otimes \cU_2 \otimes \cU_3).
\end{align*}
The situation is thus notably more complicated than the $\GL$-oriented case.
 Nevertheless, one can derive some relations amongst the coefficients $a_{ijk}^l$
 that we summarize here:
\begin{itemize}
\item \textit{Degree}: $a_{ijk}^l$ is a cohomology class in $\E^{4d,2d}(S)$, $d=l-i-j-k$. 
\item \textit{Symmetry} (Par.~\ref{num:symmetry}): for all integers $i,j,k$,
 $a_{ijk}^l=a_{jik}^l=a_{ikj}^l$.
\item \textit{Neutral identities}, (Prop.~\ref{prop:ternary_monomials}):
 when $\E$ is $\SL$-oriented, one gets
\begin{align*}
& i \neq l \Rightarrow a^l_{i00}=0, \\
& a^1_{100}=a^3_{300}=2h, \ a^2_{200}=2(h-\epsilon), \ a^4_{400}=1, \\
& {\sum}_{i=0}^r a^4_{i,r-i,0}=0,
\end{align*}
where $h$ and $\epsilon$ are the image of natural endomorphisms of the sphere spectrum
 over $S$ via the ring map $\End(\SSp_S) \rightarrow \E^{00}(S)$.\footnote{Note that
 these elements should be interpreted as the generators of the free abelian group
 $\GW(\ZZ)=\langle \epsilon,h \rangle$.
 As quadratic forms, $h=\langle 1,-1\rangle$ and $\epsilon=-\langle -1\rangle$.
 There is a canonical map $\GW(\ZZ) \rightarrow \End(\SSp_\ZZ)$ which is an isomorphism
 up to torsion.}
\end{itemize}
Our first contribution is to determine the ternary law associated
 with MW-motivic cohomology. 
\begin{thmi}[See Th.~\ref{thm:HMW_ternary} and Cor.~\ref{cor:HMW_ternary}]\label{thm:introA}
Let $R$ be a ring of coefficients and $S$ a scheme such that one of the following conditions
 hold:
\begin{enumerate}
\item[(a)] $R=\ZZ$, $S$ is a scheme over a perfect field $k$ of characteristic $\neq 2,3$.
\item[(b)] $R=\QQ$, $S$ is an arbitrary scheme.
\end{enumerate}
Then the $\SL$-oriented ring spectrum $\HMW R_S$ has the following ternary laws:
$$
F_l(x,y,z)=
\begin{cases}
2h.\sigma(x), & l=1, \\
2(h-\epsilon).\sigma(x^2)+2h.\sigma(xy), & l=2, \\
2h.\sigma(x^3)+2h.\sigma(x^2y)+8(2h-\epsilon).xyz, & l=3, \\
\sigma(x^4)-2h.\sigma(x^3y)+2(h-\epsilon).\sigma(x^2y^2)+2h.\sigma(x^2yz), & l=4.
\end{cases}
$$
\end{thmi}
Given the preceding tables of analogy,
 these ternary laws are the analog of the additive formal group law
 and we call it the \emph{additive ternary laws}.
 The theorem ultimately reduces to the case (a),
 and to the "Witt part" of MW-motivic cohomology cohomology, which represents unramified Witt cohomology.
 In this later case, we are able to compute the relevant Borel classes of triple products
 of symplectic rank $2$ bundles, based on previous computations of Levine (see \cite{LevineEuler})
 and on a geometric determination of the associated symplectic form
 (this is where we have to assume that $6 \in k^\times$; see Appendix~\ref{sec:appendix}).

\subsection*{The Borel character}

The main contribution of our paper is the construction of the symplectic analog of the Chern character,
 that we call the \emph{Borel character}.
 Recall that the former was introduced by Grothendieck as a bridge between $K_0$-theory and Chow groups,
 later extended to higher K-theory and higher Chow groups.
 In the interpretation of Riou \cite{Riou}, it appears as an isomorphism between the rational
 (homotopy invariant) K-theory spectrum and the periodized rational motivic cohomology spectrum
 (to reflect Bott periodicity of K-theory). It can also be interpreted as the unique isomorphism between
 the universal rational ring spectra respectively with multiplicative and additive formal group laws,
 reflecting now the exponential map between these formal group laws.

The method of Riou for building the Chern character uses the classical toolkit for studying cohomological
 operations in stable homotopy, as pioneered by Adams.
 In this setting, there is an important conceptual distinction between the unstable and stable operations.
 In his fundamental work, Riou did not only compute all unstable operations in algebraic K-theory,
 but he also gave criterion so that these operations can be lifted to stable ones.

In our paper, we extend Riou's method to the symplectic case, replacing K-theory by hermitian K-theory.
 Moreover, we slightly extend the domain of applicability of Riou's result by determining all possible natural
 transformation, for presheaves over the category of smooth $S$-scheme $\Sm_S$,
 between the hermitian K-theory $\KSp_0$
 and some cohomology group with coefficients in an arbitrary $\Sp$-oriented cohomology. 
 More precisely:
\begin{thmi}[see Theorems~\ref{thm:set_symplectic_op} and \ref{thm:stable_symplectic_op}]\label{thm:introB}
Let $\E$ be an $\Sp$-oriented ring spectrum over a regular $\ZZp$-scheme $S$, with Borel classes $b_i$
 and ring of coefficients $\E^{**}=\E^{**}(S)$. Let $(n,i) \in \ZZ^2$ be a pair of integers.
\begin{enumerate}
\item The following application is a bijection:
\begin{align*}
\big(E^{**}[[t_r,r \geq 1]]^{(n,i)}\big)^\ZZ
 & \rightarrow \Hom_{\mathrm{Sets}}(KSp_0,\E^{n,i}) \\
(F_r)_{r \in \ZZ}
 & \mapsto \{(\cV,r) \mapsto F_r\big(b_1(\cV),\hdots,b_r(\cV),0,\hdots\big)\}
\end{align*}
where $E^{**}[[t_r,r \geq 1]]^{(n,i)}$ denotes the formal power series with coefficients in the graded ring $\E^{**}$,
 in the indeterminate $t_r$, which are homogeneous of degree $(n,i)$, each $t_r$ being given degree $(4,2)$;
 $\cV$ is a symplectic bundle over some $X$ in $\Sm_S$, and $r$ its rank.
\item The following application is an isomorphism of bigraded abelian groups:
$$
E^{**}[[b]]
 \rightarrow \Hom_{\mathrm{Ab}}(KSp_0,\E^{*,*}), b^n \mapsto \tilde \chi_{2n}:KSp_0 \rightarrow \E^{4n,2n}
$$
where $b$ is an indeterminate whose bidegree is $(4,2)$;
 $\tilde \chi_n$ is the natural transformation defined on a symplectic bundle $\cV$ over some
 $X$ in $\Sm_S$ by the formula --- see \eqref{eq:determinant_tilde_chi}:
$$
\tilde \chi_{2n}(\cV)=
\left|
\raisebox{0.5\depth}{
\xymatrix@=0.1ex{
b_1 & 1\ar@{-}[rrrddd] & && \ar@{}|/-30pt/0[lllldddd]\\
2b_2\ar@{.}[ddd] & b_1\ar@{-}[rrrddd] && & \\
& b_2\ar@{-}[rrdd]\ar@{.}[dd] & & & \\
&  & & & 1 \\
nb_n & b_{n-1}\ar@{.}[rr] & & b_2 & b_1
}
}
\right|.
$$
\end{enumerate}
\end{thmi}
As a matter of terminology, we call the morphisms in (1) (resp. (2))
 the set (resp. groups) of \emph{unstable} (resp. \emph{additive}) \emph{symplectic operations}
 (or simply $\Sp$-operation)
 on $\E$ of bidegree $(n,i)$.
The proof follows the strategy designed by Riou, adapted to the $\Sp$-oriented case.
 We note that some form of the first assertion already appears
 in the foundational work of Panin and Walter \cite[Th. 11.4]{PW2}.

The next step is to determine \emph{stable operations}, i.e.
natural transformations of representable cohomology theories
 compatible with the stability isomorphism. In contrast with the classical situation,  it is more convenient to consider the
 (pointed) sphere $H_S:=(\HP^1_S,\infty)^{\wedge,2}=\un(4)[8]$
 in the case of $\Sp$-oriented ring spectra--- this is justified by
 the $(8,4)$-periodicity of the symplectic K-theory spectrum $\KSp$. In this setting and along
 classical lines, the stability class associated with any ring spectra $\E$ over a base scheme $S$
 is the structural class $\sigma^\E_S \in \tilde \E^{8,4}(H_S)$ in reduced cohomology
 such that for any smooth $S$-scheme $X$,
 the following exterior product map is an isomorphism:
$$
\tilde \E^{n,i}(X_+) \xrightarrow{\gamma^\E_\sigma} \tilde \E^{n+8,i+4}(H_S \wedge X_+).
$$
When $\E$ is $\Sp$-oriented, one has $\sigma^\E_S=b_1(\cU_1).b_1(\cU_2)$ where $\cU_i$ is the
 tautological symplectic bundle on the $i$-factor of $H_S$.
 Then given an additive $\Sp$-operation $\theta:\KSp_0 \rightarrow \E^{n,i}$ as above,
 we define an associated "desuspended" $\Sp$-operation $\omega_H(\theta)$
 on $\E$ of degree $(n-8,i-4)$ by the following commutative diagram:
$$
\xymatrix@R=14pt@C=44pt{
\KSp_0(X)\ar^-{\gamma^{\KSp}_\sigma}_-\sim[r]\ar_{\omega_H(\theta)}[d]
 & \tilde \KSp_0(H \wedge X_+)\ar^\theta[d] \\
\E^{n-8,i-4}(X)\ar^-{\gamma^{\E}_\sigma}_-\sim[r] & \tilde \E^{n,i}(H \wedge X_+).
}
$$
Then a stable $\Sp$-operation is nothing but a sequence $(\Theta_n)_{n \geq 0}$
 of unstable (necessarily additive) $\Sp$-operations such that $\Theta_n=\omega_H(\Theta_{n+1})$.
 Our next result is the computation of the desuspension of every $\Sp$-operations on MW-motivic cohomology:
\begin{thmi}[See Theorem~\ref{thm:desuspension_add_op_HMW}]\label{thm:introC}
Let $R$ be a ring of coefficients and $S$ a scheme such that one of the following conditions
 hold:
\begin{enumerate}
\item[(a)] $R=\ZZ$, $S$ is a scheme over a perfect field $k$ of characteristic $\neq 2,3$.
\item[(b)] $R=\QQ$, $S$ is a $\ZZp$-scheme.
\end{enumerate}
 Then for any integer $n \geq 0$, the following relation holds:
$$
\omega_H\left(\tilde \chi_{2n+4}^R\right)=\psi_{2n+4}.\tilde \chi_{2n}^R
$$
where $\tilde \chi_{2n}^R$ is the additive $\Sp$-operation on $\HMW$ of degree $(8,4)$
 obtained in Theorem~\ref{thm:introB}
 and $\psi_{2n+4}$ is the image in $\HMW^{00}(S)$ of the following quadratic form (seen in $\GW(\ZZ)$):
$$
\psi_{2n+4}=\begin{cases}
\frac 12(2n+4)(2n+3)(2n+2)(2n+1).h & \text{if $n$ is even.} \\
(2n+4)(2n+2).\big((2n^2+4n+1).h-\epsilon\big) & \text{if $n$ is odd.} \\
\end{cases}
$$
\end{thmi}
Note in particular that $\rk(\psi_{2n+4})=(2n+4)(2n+3)(2n+2)(2n+1)$; consequently this result is coherent
 with the one obtained by Riou (in \cite{Riou}).\footnote{Recall once again that hermitian K-theory is $(8,4)$-periodic
 whereas algebraic K-theory is $(2,1)$-periodic.}
The main ingredient of this computation is Theorem~\ref{thm:introA}.

As a corollary, we obtain the computation of all stable $\Sp$-operations on rational MW-motivic cohomology
 and ultimately deduce the announced construction of the \emph{Borel character}:
\begin{thmi}[See Theorem~\ref{thm:stable_symplectic_op}, \ref{thm:Borel_iso}, Par.~\ref{num:Borel_ppties}]
We assume that conditions (a) or (b) of Theorem~\ref{thm:introC} holds,
 and denote by $[-,-]$ maps in $\SH(S)$.
\begin{enumerate}
\item For any integer $n \in \ZZ$, one has canonical isomorphisms:
\begin{align*}
\Hom_{\mathrm{St}}\big(KSp_0,\wCH^{2n}_\QQ\big) 
 \simeq & \big[\KSp_S,\HMW \QQ_S(2n)[4n]\big] \\
 & \simeq \begin{cases}
 \big[\KSp_S,\HM \QQ_S(2n)[4n]\big] \simeq \QQ & n=2i, \\
  \GW(S)_\QQ=\QQ\oplus \W(S)_\QQ & n=2i+1,
\end{cases}
\end{align*}
where the left hand-side denotes the stable symplectic operations
 on rational MW-motivic cohomology (whose $(4n,2n)$-part is given by 
 rational Chow-Witt groups $\wCH^{2n}_\QQ$).
\item Define the \emph{Borel character} as the following map:
\begin{align*}
\bo_t:\KQ^\QQ_{S} \xrightarrow{\ (\bo_{2n})_{n \in \ZZ}\ }
 & \bigoplus_{n \text{ even}} \HMW \QQ_S(2n)[4n] \oplus \bigoplus_{n \text{ odd}} \HM \QQ_S(2n)[4n],
\end{align*}
where $\bo_{2n}$ is the stable operation
$$
\GW_S \simeq \KSp_S(-2)[-4] \rightarrow \HM \QQ_S(2n)[4n]
$$
corresponding to the unit in $\QQ$ (resp. $\GW(S)_\QQ$) under the above isomorphism if $n$ is odd (resp. even).
 Then $\bo$ is an isomorphism of ring spectra and the following diagram commutes:
$$
\xymatrix@C=24pt@R=14pt{
\KQ^\QQ_{S}\ar^-{\bo_t}[r]\ar_f[d] & \bigoplus_{n \text{ even}} \HMW \QQ_S(2n)[4n] \oplus \bigoplus_{n \text{ odd}} \HM \QQ_S(2n)[4n]\ar[d] \\
\KGL^\QQ_{S}\ar^-{\ch_t}[r] & \bigoplus_{m \in \ZZ} \HM \QQ_S(m)[2m].
}
$$
\end{enumerate}
\end{thmi}
The right vertical map is obtained by forgetting the Witt part in degrees 0 modulo 4, while the left-hand vertical map is the forgetful functor.
%
In fact, the Borel character is the sum of the Chern character in even degrees and a "Witt part",
 concentrated in degrees 0 modulo 4. Note moreover that, from the results of \cite{DFJK1}, the Witt part
 is only visible on the characteristic zero part of the scheme $S$.
 This is because it exists, as a morphism of ring spectra, only modulo torsion.
 
To conclude this introduction, let us mention that even though the Borel character is a stable and rational phenomenon,
 the methods and results of this paper apply more generally to unstable and integral situations. Indeed, we note that for a given scheme $X$,
 one needs only to invert finitely many quadratic forms $\psi_{2n+4}$ appearing in Theorem~\ref{thm:introC} in order to get the Borel character.
 Presumably, this yields finer results as simply tensoring with $\QQ$.
 This is somewhat illustrated by the following theorem, which shows that Milnor-Witt $K$-theory embeds into suitable hermitian $K$-theory groups modulo torsion.
\begin{thmi}[see Theorem~\ref{thm:strongsuslin}]\label{thm:introE}
For any $n\geq 2$ and any finitely generated field extension of $k$, the composite
\[
\KMW_n(L) \xrightarrow{\varepsilon_{n,n}} \GW_n^n(L) \xrightarrow{\tilde\chi_{n,n}} \KMW_n(L)
\]
is multiplication by $\psi_{\mu(n)}! \in \GW(k)$, where $\mu(n)$ is the smallest integer of the form $2+4r$ greater than $n$
 and we have put:
$$
\psi_{2+4r}!:=\psi_2 \cdot \psi_6 \cdot \hdots \cdot \psi_{2+4r}.
$$
\end{thmi}
 This result can be seen as a generalization of a theorem of Suslin (\cite{Suslin84}),
 stating that Milnor $K$-theory embeds into (Quillen) $K$-theory modulo torsion.

\subsection*{Linked and further works}

Our result on computing the ternary laws of MW-motivic cohomology owes much to the reading of \cite{LevineEuler},
 as the reader will see in the text.
 The results of Ananyevskiy's thesis, published in \cite{Anan1}, are especially linked to the results obtained here.
 Indeed, Ananyevskiy computes the ternary laws
 (without the abstract theory explained here)
 associated with higher Witt groups in Lemma 8.2 of \emph{op. cit.}
 and build the minus part of our Borel character in Theorem 1.1 of \emph{op. cit.}.

We plan to come back to explicit computations of the Borel character in future work,
 in collaboration with Fangzhou Jin and Adeel Khan.
 Especially, we will study the natural question that arises with the analogy
 between the Chern and Borel characters: finding the quadratic analog of the Grothendieck-Riemann-Roch formula.
 We also plan to compute the ternary laws associated to some well-known cohomology theories, such as higher Grothendieck-Witt groups.

\subsection*{Plan} \label{plan}

In Section~\ref{sec:worientations}, we recall Panin and Walter's theory of generalized orientations, Borel classes
 and introduce Walter's notion of ternary laws.

In Section~\ref{sec:ternary}, we compute the ternary laws associated with MW-motivic cohomology,
 as explained in Theorem~\ref{thm:introA} above.
 The proof reduces to compute the relevant Borel classes
 either in Chow groups or in Witt-cohomology. The last part is the core of the proof and occupies Section~\ref{sec:Witt},
 complemented with Appendix~\ref{sec:appendix} containing an "elementary" computation of some threefold tensor product
 of symplectic bundles which is central in our computations.

Section~\ref{sec:operations} is devoted to the implementation of Riou's method for determining cohomological operations
 in the symplectic case. The first subsection is devoted to prove Theorem~\ref{thm:introB}. 
 The second subsection gives some abstract considerations to determine stable $\Sp$-operations in the general case.
 The core of the section is the third subsection which computes the obstruction to stabilization
 for MW-motivic cohomology, as explained in Theorem~\ref{thm:introC} above.

In Section~\ref{sec:Borel}, we define the Borel character (Definition~\ref{df:Borel_char}) and prove that it is an isomorphism
 of ring spectra. On the model of the proof of Theorem~\ref{thm:introA}, we treat the plus part and minus part separately.
 The plus part can be reduced to the classical case of the Chern character, while the minus part can be treated
 using properties of periodic ring spectra as recalled in Subsection~\ref{sec:periodic} and the ideas of \cite{ALP} 
 suitably extended to arbitrary base $\ZZp$-schemes.

Section~\ref{sec:suslinhomomorphism} contains the proof and statement of Theorem~\ref{thm:introE},
 based on ideas of \cite{Asok14b}.

\section*{Acknowledgments}

The authors are grateful to F. Jin and A. Khan for discussions during our collaboration
 on absolute purity theorems \cite{DFJK1} that were linked with the present work. The second author warmly thanks I. Panin for a very interesting conversation on the analogue of the Chern character linking Grothendieck-Witt groups and Chow-Witt groups. This conversation led to many considerations appearing in this work.
 We express our gratitude to C. Walter for sharing with us his ideas and a preliminary work on ternary laws.
 We also want to thank A. Ananyevskiy, A. Asok, J. Hornbostel, M. Levine and M. Wendt 
 for discussions around the subject of this paper.

F. D\'eglise received support from the French "Investissements d'Avenir" program,
 project ISITE-BFC (contract ANR-lS-IDEX-OOOB).

\section*{Notation and conventions} \label{conventions}

We will fix a base scheme $B$, which in practice will be
 the spectrum of either a perfect field denoted by $k$,
 the ring $\ZZp$ or $\ZZ$.
 We work with the category $\Sch_B$ of quasi-compact and quasi-separated $B$-schemes;
 all schemes are supposed to be in $\Sch_B$. For certain results,
 we will also restrict our attention to regular finite dimensional $B$-schemes,
 and we denote by $\Reg_B$ the corresponding category.
 
Unless explicitly stated, we will consider \emph{(ring) spectra} $\E$ over $B$,
 and look at them as absolute (ring) spectra over $\Sch_B$
 by putting $\E_X=f^*\E$ for any $f:X \rightarrow B$.
 Here are the examples that will appear in the present paper:

\underline{The case $B=\Spec(\ZZ)$}: the absolute ring spectra $\HM \ZZ$
 (resp. $\HM \QQ$, $\KGL$),
 representing integral motivic cohomology (resp. rational motivic cohomology, homotopy invariant K-theory). 
 See \cite{Spit} (resp. \cite{CD3}, \cite{CisKth}). \\
 We will also define the rational Milnor-Witt motivic ring spectrum as 
$$
\HMW \QQ:=\SSp_\QQ,
$$
where $\SSp$ is the motivic sphere spectrum. We refer the reader to \cite[Def. 6.1]{DFJK1}
 for a better definition. It coincides with the above one according to  Cor. 6.2
 of \emph{op. cit.}
 Note also that according to \cite[Cor. 8.9]{DFJK1}, one has for any regular scheme S:
\begin{equation}\label{eq:MWQ&wChow}
\HMW^{2n,n}(S,\QQ)\simeq \wCH^{n}(X)_\QQ=\CH^n(X)_\QQ \oplus H^n_\mathrm{Zar}(S_\QQ,\uW)
\end{equation}
 where $\uW$ is the Zariski sheaf over $S_\QQ$ associated to the Witt functor.

\underline{The case $B=\Spec(\ZZp)$}: the absolute ring spectra $\KQ$ (resp. $\KW$), representing 
 higher Grothendieck-Witt groups \cite[\textsection 9]{Schlicht2} (also called hermitian K-theory)
 (resp. Balmer's derived Witt-groups) over regular schemes. \\
 For the definition of $\KQ$ we refer the reader to \cite{PW1} and
 \cite{ST}.\footnote{Beware that the spectrum $\GW$ is also denoted by $\mathrm{KO}$
 or $\mathrm{KQ}$ in the literature. We follow here the notation of \cite{ST}.}
 To fix our conventions, let us recall that for a regular scheme $S$:
\begin{equation}\label{eq:HGW&KO-KSp}
\KQ^{n,i}(S)=GW_{2i-n}^i(S)=\begin{cases}
KO_{2i-n} & i \cong 0 \mod 4, \\
KSp_{2i-n} & i \cong 2 \mod 4,
\end{cases}
\end{equation}
where $KO_*$ (resp. $KSp_*$) denotes the higher hermitian K-theory of orthogonal bundles
 (resp. symplectic bundles) with the canonical duality (see again \cite{Schlicht2}).
 For the definition of $\KW$, we refer to \cite[Def. 3]{ALP}: $\KW=\KQ[\eta^{-1}]$
 where $\eta$ is the (algebraic) Hopf map.

\underline{The case $B=\Spec(k)$, $k$ perfect field of characteristic not $2$}: the absolute ring spectrum $\HMW \ZZ$
 representing integral Milnor-Witt cohomology as defined in \cite{DF2}.
 We will also consider $\HH \uKMW$ (resp. $\HH \uW$)
 the spectrum associated with the unramified Milnor-Witt K-theory (resp. Witt-theory),
 which represents Chow-Witt groups (resp. unramified Witt-cohomology)
 according to \cite{MorLNM}. In particular, we have for any smooth $k$-scheme $S$:
\begin{equation}\label{eq:MWZ&wChow}
\HMW^{2n,n}(S) \simeq H_\mathrm{Zar}(S,\uKMW) \simeq \wCH^n(S).
\end{equation}
Finally, we will use Morel's plus/minus decomposition of the $\ZZp$-linear stable homotopy category
 (see eg \cite[16.2.1]{CD3} or \cite[Rem. 4]{ALP}). Recall in particular the identifications
 (see \cite[16.2.13]{CD3} for the second one):
\begin{equation}\label{eq:MW&M_plusminus}
\begin{split}
\HMW \QQ_S&=\SSp_{S,\QQ+} \oplus \SSp_{S,\QQ+} \\
\HM \QQ_S & \simeq \SSp_{S,\QQ+}.
\end{split}
\end{equation}

\section{Weak orientations}\label{sec:worientations}

\subsection{Definitions and basic properties}

\begin{num}
Given a vector bundle $V$ over a scheme $X$, we let $\Th(V)=V/V^\times$ be its Thom space in $\SH(X)$.
 Given a spectrum $\E$, and integers $(n,i)\in \ZZ^2$, we put as usual:
$$
\E^{n,i}\big(\Th(V)\big)=\Hom_{\SH(X)}(\Th(V),\E_X(i)[n]).
$$
Recall that the Thom space functor is a monoidal functor with respect to direct sums of vector bundles and
 tensor products of spectra: $\Th(E \oplus F) \simeq \Th(E) \otimes \Th(F)$.
 In particular, tensor products of morphisms induce an exterior product:
$$
\E^{n,i}\big(\Th(E)\big) \otimes \E^{m,j}\big(\Th(F)\big) \rightarrow \E^{m+n,i+j}\big(\Th(E \oplus F)\big), x \otimes y \mapsto x \cdot y.
$$
\end{num}

\begin{num}\label{num:pre_w_or}
We will consider the following $\ZZ$-graded algebraic sub-groups of $\GL_*$:
\begin{equation}\label{eq:alg_gps}
\Sp_* \rightarrow \SL_* \rightarrow \SL^c_* \rightarrow \GL_*.
\end{equation}
The $n$-th graded part of $\Sp_*$ is $\Sp_{2n}$ so that the first map is homogeneous
 of degree $2$. All the other maps are homogeneous of degree $1$.
 They are all classical algebraic groups except $\SL^c_*$ which was introduced in \cite[3.3]{PW1}. 
 Recall that its $n$-th graded part is defined\footnote{We follow the convention of \cite[Rem. 2.8]{Anan19}} as the kernel of the map
$$
\GL_n \times \GG \rightarrow \GG, (g,t) \mapsto t^{-2}\det(g).
$$
Letting $G=G_*$ be one of these groups, there is a classical correspondence between G-torsors over a given scheme $X$
 and vector bundles $E$ over $X$ equipped with extra structures. We summarize the situation in the following table:
\begin{center}
\begin{tabular}{|c|c|}
\hline
group G & bundle $E$ with extra structure \\
\hline
$\SL^c_*$ & $(E,\lambda,L)$, $L$ line bundle, $\lambda:\det E \xrightarrow \sim L^2$ \\
\hline
$\SL_*$ & $(E,\lambda)$, $\lambda:\det E \xrightarrow \sim \AA^1_X$ \\
\hline 
$\Sp_*$ & $(E,\psi)$, $\psi:E \otimes E \rightarrow \AA^1_X$ symplectic form \\
\hline
\end{tabular}
\end{center}
Recall that our symplectic forms are always assumed to be non-degenerate. For short,
 we will say $G$-bundle for a bundle equipped with the corresponding extra structure.
 Morphisms of $G$-bundles over $X$ are morphisms of vector bundles over $X$ which preserve the extra structure in the obvious sense.

One can check that the category $\V(X,G)$ of $G$-bundles over $X$ is additive.
 Let $r=2$ (resp. $r=1$) when $G=\Sp_*$ (resp. in the other cases).
 The constant bundle of rank $r$ admits a canonical $G$-bundle structure and the corresponding object is denoted by $\un^G_X$.
 In particular, given a (ring) spectrum $\E$ over $X$, one gets an isomorphism
$$
\sigma:\E^{2r,r}\big(\Th(\un_X^G)\big) \simeq \E^{0,0}(\un_X).
$$
The following definition is a slight extension of the original one due to Panin and Walter
 (our reference text will be \cite{Anan19}).
\end{num}
\begin{df}\label{df:worientations}
Consider the above notations. 
 Let $\E$ be a ring spectrum over $B$, with unit $1_X$ over a $B$-scheme $X$. 
 An absolute $G$-orientation $\thom$ of $\E$ is the data for
 every $G$-bundle $\cV$ of rank $r$ over a scheme $X$ in $\Sch_B$ of a class $\thom(\cV) \in \E^{2r,r}(\Th(\cV))$,
 satisfying the following properties:
\begin{itemize}
\item \textit{Isomorphisms compatibility}. $\phi^*\thom(\cV)=\thom(\cW)$ for $\phi:\cV \xrightarrow \sim \cW$.
\item \textit{Pullbacks compatibility}. $f^*\thom(\cV)=\thom(f^{-1}\cV)$ for $f:Y \rightarrow X$ in $\Sch_B$.
\item \textit{Products compatibility}. $\thom(\cV \oplus \cW)=\thom(\cV) \cdot \thom(\cW)$.
\item \textit{Normalization}. $\thom(\un_X^G)$ corresponds to $1_X$ via the isomorphism $\sigma$.
\end{itemize}
In this situation, we also say that $(\E,\thom)$ is absolutely \emph{$G$-oriented}.
\end{df}
This definition is tailored to generalize the classical notion of orientation in motivic homotopy theory,
 which corresponds in the above term to a $\GL$-orientation.

\begin{rem}
When $\Sch_B$ is the category of smooth $B$-schemes of finite type,
 we will simply say that $\E$ is $G$-oriented, and $\thom$ is a $G$-orientation.
 This is the original definition of \cite{PW1} and \cite{Anan19}. \\
 When $G=GL_*$ (resp. $G=\Sp_*$), a $G$-orientation uniquely determines an absolute $G$-orientation.
 Indeed, according to \cite{Deg12} (resp. Panin-Walter's symplectic projective bundle formula, see recall in \ref{thm:sp_pb_thm}),
 a $G$-orientation over a scheme $X$ is determined by a class $c \in \tilde \E^{2,1}(\PP^\infty_B)$
 (resp. $b \in \tilde \E^{4,2}(\HP^\infty_B)$) whose restriction to $\PP^1_B$
 (resp. $\HP^1_B$) is the suspension of $1_X$.
 Thus a $G$-orientation of $\E$ over $B$ induces by pullback a $G$-orientation over any $B$-scheme $X$,
 and therefore an absolute $G$-orientation. \\
On the other hand, we do not know an analogue statement for $\SL$ or $\SL^c$ orientations.
\end{rem}

\begin{rem}\label{rem:link_orientations}
As proved in \cite[Lem. 1.4]{Anan19}, one deduces from the above definition
 that given a map $\phi:G \rightarrow H$ from Diagram \eqref{eq:alg_gps},
 an $H$-orientation $\thom$ on a ring spectrum $\E$ induces a $G$-orientation
 on $\E$, namely $\thom \circ \phi_*$, where $\phi_*:\V(X,G) \rightarrow \V(X,H)$
 is the additive functor induced by $\phi$. We state the following relations between orientations for further use:
$$
\text{orientation}
 \Rightarrow \text{$\SL^c$-orientation}
 \Rightarrow \text{$\SL$-orientation}
 \Rightarrow \text{$\Sp$-orientation}.
$$
\end{rem}

\begin{ex}\label{ex:orientations}
Here are examples among the spectra that will appear in this work.
\begin{itemize}
\item The absolute ring spectra $\HM R$ and $\KGL$ are oriented (this is classical, see e. g. \cite{Deg12}).
\item The absolute ring spectra $\KQ$ and $\KW$ are $\SL^c$-oriented. The case of hermitian K-theory
 is proved in \cite[Th. 5.1]{PW1}. The case of derived Witt groups follows from that of hermitian
 K-theory, given Definition \ref{df:KW} which gives a morphism of absolute ring spectra $\KQ\rightarrow \KW$.
\end{itemize}
\end{ex}

\begin{df}\label{df:Euler}
Let $(\E,\thom)$ be a $G$-oriented ring spectrum as in the previous definition.

Then for any scheme $X$, and any $G$-bundle $\cV$ with underlying vector bundle $V$ of rank $r$,
 we define the associated Euler class by the formula:
$$
e^\thom(\cV)=s^*\pi^*\thom(\cV) \in \E^{2r,r}(X)
$$
where $s:X \rightarrow V$ (resp. $\pi:V \rightarrow \Th(V)$) is the zero section (resp. canonical projection map).
\end{df}
When the orientation $\thom$ is clear from the context, we simply write $e(\cV)$.
 We will apply the same convention for all the other characteristic classes associated with $\Sp$-orientations. This is harmless in this work as we will never consider two different such orientations
 on our ring spectra.

One immediately deduces from the properties of Thom classes the following (usual) properties of Euler classes.
\begin{prop}
Consider the assumptions of the previous definition. Then Euler classes satisfy the following formulas:
\begin{itemize}
\item \textit{Invariance under isomorphisms}. $\phi^*e(\cV)=e(\cW)$ for $\phi:\cV \xrightarrow \sim \cW$.
\item \textit{Pullbacks compatibility}. $f^*e(\cV)=e(f^{-1}\cV)$ for $f:Y \rightarrow X$ in $\Sch_B$.
\item \textit{Products compatibility}. $e(\cV \oplus \cW)=e(\cV) \cdot e(\cW)$.
\item \textit{Vanishing}.  $e(\cV)=0$ whenever $\cV$ contains a trivial $G$-bundle as a direct factor.
\end{itemize}
\end{prop}


%
%

\subsection{Borel classes}

\begin{num}\label{num:basic_sp_bdl}\textit{Symplectic bundles}.
Let us be more specific about symplectic vector bundles,
 introduced in Paragraph \ref{num:pre_w_or}.

First recall that any vector bundle $V/X$ admits a {\it symplectification}:
$$\cH(V):=\left(V \oplus V^\vee,\begin{pmatrix} 0 & 1 \\ -\mathrm{can} & 0 \end{pmatrix}\right),
$$
where $\mathrm{can}:V\to V^{\vee\vee}$ is the usual canonical isomorphism.
Using the notation of \emph{loc. cit.}, we get in particular:
 $\un^\Sp=\cH(\AA^1)$, that we will simply denote by $\cH$ in the sequel. To comply with the classical notations,
 the direct sum of symplectic vector bundles will be denoted by $\perp$.
 In particular, $\cH(\AA^n)=\cH^{\perp,n}$.
 More generally, $\cH$ sends $\oplus$ to $\perp$.
\end{num}

\begin{num}\label{num:sp_pb_thm}
As explained by Panin and Walter,
 $\Sp$-oriented ring spectra are analogous to ($\GL$-)oriented ones.
 Let us first recall the $\Sp$-projective bundle theorem.
 Consider a $\Sp$-bundle $(V,\phi)$ over a scheme $X$.
 Panin and Walter in \cite{PW2} introduced the projective $\Sp$-bundle $\HP(V,\phi)$
 as the open subscheme of the Grassmannian scheme $\mathrm{Gr}(2,V)$ on which the restriction of $\phi$ to the canonical sub-bundle of rank $2$ is non-degenerate.

We let $U$ be the tautological rank $2$ bundle on $\HP(V,\psi)$. By definition, it is equipped with a
 symplectic structure $\psi$ coming from the restriction of $\phi$ and we set $\mathfrak U=(U,\psi)$.

The following $\Sp$-projective bundle theorem is due to Panin and Walter; see \cite[Th. 8.2]{PW2} for a proof.
\end{num}
\begin{thm}\label{thm:sp_pb_thm}
Consider the above notations and assume that $V$ has rank $2n$.
 Let $p:\HP(V,\psi) \rightarrow X$ be the canonical projection, $\E$ be an an $\Sp$-oriented ring spectrum
 and $b=e(\cU)$ be the associated Euler class (Def. \ref{df:Euler}).

Then the following map is an isomorphism of bi-graded $\E^{**}(X)$-modules:
$$
\bigoplus_{i=0}^n \E^{**}(X) \rightarrow \E^{**}\big(\HP(V,\psi)\big),
 x_i \mapsto p^*(x_i).b^i.
$$
\end{thm}

\begin{num}\textit{Borel classes}.\label{num:Borel_classes}
 It is easy to derive from the previous theorem the theory of
 Borel classes. Under the above notation,
 they are the classes $b_i(V,\psi) \in \E^{4i,2i}(X)$ for $i \geq 0$ uniquely determined
 by the following relations:
\begin{equation}\label{eq:borel_pb_formula}
\sum_{i=0}^n (-1)^i.b_i(V,\psi).b^{n-i}=0, b_0(V,\psi)=1, \forall i>n, b_i(V,\psi)=0.
\end{equation}
In addition, they satisfy the following relations:
\begin{enumerate}
\item $\rk(V)=2n$, $b_n(V,\psi)=e(V,\psi)$.
\item \textit{Invariance under isomorphisms}. $b_i(\cV)=b_i(\cW)$ for $\phi:\cV \xrightarrow \sim \cW$.
\item \textit{Pullbacks compatibility}. $f^*b_i(\cV)=b_i(f^{-1}\cV)$ for $f:Y \rightarrow X$ in $\Sch_B$.
\item \textit{Trivial bundles}. $b_i\big[\cH(\AA^n)\big]=0$ for $i>0$, $n>0$.
\item \textit{Whitney sum formula}. $b_t(\cV \perp \cW)=b_t(\cV) \cdot b_t(\cW)$ in $\E^{**}(X)[t]$,
 where $b_t$ denotes the {\em total Borel class}:
\begin{equation}\label{eq:Borel_total}
b_t(\cV)=\sum_i b_i(\cV).t^i.
\end{equation}
\end{enumerate}
One can reformulate the two last properties by saying that the total Borel class
 $b_t$ factors through the $0$-th symplectic $K$-theory group, $\KSp_0(X)$,
 and actually induces a morphism of abelian groups, sending sums to products:
\begin{equation}\label{eq:total_Borel}
b_t:\KSp_0(X)/\ZZ\big[\cH\big] \rightarrow \E^{**}(X)[t]^\times.
\end{equation}
\end{num}

\begin{num}\label{num:sp-splitting-principle}
\textit{The symplectic splitting principle}. As for the classical Chern classes (see eg. \cite[Section 3.2]{Ful}),
 one derives from the previous theorem a splitting principle.
 Given any symplectic bundle $\cV$ over a scheme $X$, there exists an affine morphism $p:X' \rightarrow X$
 inducing a monomorphism $p^*:\E^{**}(X) \rightarrow \E^{**}(X')$ with the property that $p^{-1}(\cV)$ splits
 as a direct sum of rank $2$ symplectic bundles: $p^{-1}(\cV)=\mathcal X_1 \perp \hdots \perp \mathcal X_n$.

One defines the {\it Borel roots} of $\cV$ as the Borel classes $\xi_i=b_1(\mathcal X_i)$ so that
 by the preceding Whitney sum formula, the Borel classes of $\cV$ are the elementary symmetric polynomials
 in the variables $\xi_i$.

As in the classical case, one can compute universal formulas involving Borel classes of $\cV$ by introducing Borel roots
 $\xi_i$, which reduces to rank $2$ symplectic bundles, compute as if $\cV$ was completely split, and
 then express the resulting formula in terms of
 the elementary symmetric polynomials in the $\xi_i$. This principle will be used repeatedly in Section \ref{sec:ternary}.
\end{num}

A particular instance of Borel classes will be useful in the sequel (see Section \ref{sec:Witt}).
\begin{df}\label{df:pontryagin}
 Consider the notations of the previous paragraph.
 Given an arbitrary vector bundles $V$ over a scheme $X$, and an integer $i\geq 0$,
 one defines its $i$-th {\em Pontryagin class} associated with the $\Sp$-oriented ring
 spectrum $\E$ as:
$$
p_i(V)=b_i\big(\cH(V)\big).
$$
\end{df}
Beware that we do not follow here the conventions of \cite[Def. 7]{Anan2}
 for which one uses a different numbering and sign: $p_i(V)=(-1)^i.b_{2i}\big(\cH(V)\big)$.

\begin{num}\label{num:action_units}
For the next formula, we need some notation. Let $u \in \mathcal O_X(X)^\times$ be a global unit on a scheme $X$.
 Let us consider the isomorphism $\gamma_u:\AA^1_X \rightarrow \AA^1_X$ obtained by multiplication by $u$.
 It induces a morphism of Thom spaces: $\gamma_u:\Th(\AA^1_X) \rightarrow \Th(\AA^1_X)$.
 Following Morel (eg. \cite[6.1.3]{MorelSphere}),
 we denote by $\langle u \rangle$ the corresponding automorphism of $\un_X$ in the stable homotopy category over $X$.
 These elements satisfy the following formulas in the group $\End(\un_X)$:
\begin{itemize}
\item[(U1)] $\langle u \rangle \cdot \langle v \rangle=\langle uv \rangle$
\item[(U2)] $\forall i, \langle u^{2i} \rangle=1$.
\end{itemize}
The following proposition is mainly due to \cite{Anan19}. At present, we don't know if it is true for a general $\Sp$-oriented spectrum.
\end{num}
\begin{prop}\label{prop:units&Borel}
Let $(\E,\thom)$ be an $\SL$-oriented ring spectrum,
 $u \in \mathcal O_X(X)^\times$ be a unit, and  $i \geq 0$ an integer.
 We write $\thom$ for the induced $\Sp$-orientation on $\E$ (Remark \ref{rem:link_orientations}).
\begin{enumerate}
\item For any $\Sp$-bundle $(V,\psi)$ of rank $2$: $\thom(V,u.\psi)
 =\langle u \rangle.\thom(V,\psi)$.
\item For any $\Sp$-vector bundle $(V,\psi)$ over $X$: $b_i(V,u\psi)=\langle u^i\rangle.b_i(V,\psi)$.
\end{enumerate}
Here, $u\psi$ is simply the composition $\gamma_u \circ \psi$.
\end{prop}
\begin{proof}
Let then $(V,\psi)$ be symplectic bundle.
 The symplectic form $\psi$ corresponds to a trivialization of the determinant $\psi:\det V\simeq \AA^1_X$,
 while the symplectic form $u\psi$ corresponds to the composite 
\[
\det V\stackrel \psi\to \AA^1_X\stackrel {\cdot u}\to \AA^1_X.
\]
The claim now follows from \cite[Lemma 7.3]{Anan19}.
 Let us now consider point (2). We denote by $(U,\psi)$ the tautological rank $2$ bundle over $\HP(V,\psi)$
 (Paragraph \ref{num:sp_pb_thm}).
 A straightforward computation shows that the tautological rank $2$ bundle on $\HP(V,u\psi)$
 is precisely $(U,u\psi)$ so that  $e(U,u\psi)=\langle u\rangle e(U,\psi)$ by (1).
 Applying the projective bundle formula \eqref{eq:borel_pb_formula} for $(V,\psi)$ and $(V,u.\psi)$ respectively,
 with $b=e(U,\psi)$, we get:
\begin{align*}
b^n&=b_1(E,\psi)b^{n-1}-\ldots-(-1)^nb_n(E,\psi), \\
\langle u^n\rangle b^n&=\langle u^{n-1}\rangle b_1(E,u\psi)b^{n-1}-\ldots-(-1)^nb_n(E,u\psi).
\end{align*}
Now we can multiply the first equation by $\langle u^n\rangle$ and identify each
 terms with the ones of the second equation. Taking into account formulas (U1) and (U2), one easily concludes.
\end{proof}

\begin{rem}\label{rem:notations_GW}
Extending the notations of Paragraph \ref{num:action_units},
 it will be convenient to introduce the following notation\footnote{They coincide with the classical notations
 for the Grothendieck-Witt ring through the isomorphism of Morel $\Aut(\un_k)=\GW(k)$ when $X=\Spec(k)$
 is the spectrum of a field}:
\begin{itemize}
\item $\langle u_1,...,u_n \rangle=\langle u_1\rangle + \hdots + \langle u_n \rangle$.
\item $\epsilon=-\langle -1 \rangle$.
\item $h=\langle 1,-1 \rangle$.
\end{itemize}
Extending a classical terminology from the theory of quadratic forms,
 we will say that a cohomology class $x \in \E^{n,i}(X)$ is \emph{hyperbolic}
 if it is of the form:
$$
x=h.x'
$$
using the $\End(\un_X)$-module structure on $\E^{**}(X)$.
\end{rem}

%
%

\subsection{Walter's ternary laws}

\begin{num}\label{num:df:ternary}
It appears clearly from the previous section that Borel classes
 are to $\Sp$-oriented spectra what Chern classes are to oriented spectra.
 It is therefore natural to look for an analogue of the the theory of formal
 group law associated to any oriented ring spectrum.

As observed by Walter\footnote{We refer here to unpublished notes of Walter.
 Walter does his computations for smooth $k$-schemes but as shown below,
 one can work with arbitrary schemes.}, the first technical problem that arises is the fact
 that symplectic bundles are not stable under tensor product. However,
 if you consider an odd number of symplectic bundles, then their tensor product
 is equipped with a canonical symplectic form -- the tensor product of the 
 symplectic forms of each bundle.
The second technical problem is that a triple product of symplectic bundles of rank $2$
 is of rank $8$. This means we will have to take into account four non trivial Borel classes
 $b_i$ for $i=1,\hdots, 4$. Once taking into account these differences, 
 we can mimic the construction of the associated formal group law
 in classical orientation theory as follows.

Let us fix an $\Sp$-oriented ring spectra $(\E,\thom)$ over an
 arbitrary scheme $X$ (Def. \ref{df:worientations}).
 Put $\E^{**}=\E^{**}(X)$.

 We consider the following triple product as an ind-smooth $X$-scheme:
 $P=\HP^\infty_X \times_X \HP^\infty_X \times_X \HP^\infty_X$.
 For $i=1, 2, 3$, let $\cU_i$ be the respective canonical rank $2$ symplectic
 bundle on the $i$-th coordinate, and let $x,y,z=b_1(\cU_1), b_1(\cU_2), b_1(\cU_3)$.
 According to the symplectic projective bundle theorem \ref{thm:sp_pb_thm},
 one gets the following isomorphism of $\E^{**}$-bigraded rings:
$$
\E^{**}(P) \simeq \E^{**}[[x,y,z]],
$$
the ring of power series in three variables. The next definition follows
 considerations initiated by Walter.
\end{num}
\begin{df}\label{df:ternary}
Consider an $\Sp$-oriented spectrum $\E$.
 We respectively associate to $\E$ the \emph{$l$-th ternary laws} for $l \in \{1,2,3,4\}$
 and the total ternary law:
\begin{align*}
F_l(x,y,z)&=b_l(\cU_1 \otimes \cU_2 \otimes \cU_3) \in \E^{**}[[x,y,z]], \\
F_t(x,y,z)&=b_t(\cU_1 \otimes \cU_2 \otimes \cU_3) \in \E^{**}[[x,y,z]][t].
\end{align*}
We will generically denote the $l$-th ternary law as:
$$
F_l(x,y,z)=\sum_{i,j,k} a_{ijk}^l.x^iy^jz^k.
$$
\end{df}
According to this definition, the bidegree of $a_{ijk}^l$, as an element
 of the ring of coefficients $\E^{**}$ is given by:
\begin{equation}\label{eq:aijkl_deg}
\deg(a_{ijk}^l)=(l-i-j-k).(4,2).
\end{equation}
As in the case of formal group laws associated with oriented ring spectra,
 the ternary laws play a universal role:
 given any rank $2$ symplectic bundles $\cV_1, \cV_2, \cV_3$ over a scheme $X$,
 one always gets the computation:
$$
b_l(\cV_1 \otimes \cV_2 \otimes \cV_3)=F_l(b_1(\cV_1),b_1(\cV_2),b_1(\cV_3)),
$$
the substitution being legitimate as the Borel classes $b_1(\cV_i)$ are nilpotent.

\begin{num}\label{num:symmetry}
It is legitimate to look for an analogue of the axioms satisfied by formal group laws
 in the case of ternary laws.
 Commutativity is obvious, as the tensor product of three symplectic bundles is commutative.
 Explicitly, the coefficient $a^l_{ijk}$ above is independent of the order of the indices $i,j,k$:
\begin{equation}
a_{ijk}^l=a_{jik}^l=a_{ikj}^l.
\end{equation}
In other words, the power series $F_l(x,y,z)$ are symmetric in the variables $x,y,z$.
 To give examples of ternary laws, it is useful to consider a basis for the symmetric polynomials
 in $x,y,z$. We choose the monomial basis denoted by:
\begin{equation}\label{eq:symmetric_mon}
\sigma(x^iy^jz^k)=\sum_{(a,b,c)} x^ay^bz^c
\end{equation}
where the sum runs over the monomials $x^ay^bz^c$ in the orbit
 of $x^iy^jz^k$ under the action of the permutations of the variables $x,y,z$. 
So taking into account the commutativity constraints, ternary laws can be written
$$
F_l(x,y,z)=\sum_{i \geq j \geq k} a_{ijk}^l.\sigma(x^iy^jz^k).
$$
\end{num}

The analogue of the relation $F(x,0)=x$
 is already more involved as shown by the following formula due to Walter.
\begin{prop}\label{prop:ternary_monomials}
Let $\E$ be an $\SL$-oriented ring spectrum over a scheme $X$.
 We apply the previous definition to the induced $\Sp$-orientation on $\E$
 (Remark \ref{rem:link_orientations}). Then one gets the following
 computations:
\begin{enumerate}
\item Using the notation of \ref{rem:notations_GW}
 and the $\Aut(\un_X)$-module structure on $\E^{**}(X)$:
$$
F_t(x,0,0)=1+2h.xt+2(h-\epsilon).x^2t^2+2h.x^3t^3+x^4t^4.
$$
In terms of coefficients:
\begin{align*}
&i \neq l \Rightarrow a^l_{i00}=0, \\
&a^1_{100}=a^3_{300}=2h, a^2_{200}=2(h-\epsilon), a^4_{400}=1.
\end{align*}
\item $F_4(x,x,0)=0$.

In terms of coefficients: $\sum_{i=0}^r a^4_{i,r-i,0}=0$.
\end{enumerate}
\end{prop}
\begin{proof}
Consider the first assertion.
We write $\cU=(U,\psi)$ for $\cU_1$.
Given that $b_i(\cH)=0$ for $i>0$ (Paragraph \ref{num:Borel_classes}),
 we get:
$$
F_l(x,0,0)=b_l\big((U_,\psi) \otimes \cH \otimes \cH\big).
$$
We easily obtain an isomorphism of symplectic bundles:
$$
\cU \otimes \cH \otimes \cH \simeq \cU \otimes \cH(\AA^2)
 \simeq (U,\psi)^{\oplus,2} \oplus (U,-\psi)^{\oplus,2}.
$$
According to Proposition \ref{prop:units&Borel}, one gets:
$$
b_t(U,-\psi)=1-(\epsilon x).t.
$$
Therefore, the Whitney sum formula (again Paragraph \ref{num:Borel_classes}) gives:
$$
b_t\big((U_,\psi) \otimes \cH \otimes \cH\big)=(1+x.t)^2.\big(1-(\epsilon x).t\big)^2
$$
so that expanding the last term gives the desired result.

To get the second point, we have to compute $b_4(\cU \otimes \cU \otimes \cH)$.
 We claim that the bundle $\cU \otimes \cU \otimes \cH$ has a nowhere vanishing section.
 Indeed, consider the short exact sequence
\[
0\to \det(\cU)\to \cU \otimes \cU\to \mathrm{Sym}^2(\cU)\to 0.
\]
Since $\cU$ is symplectic, we have an isomorphism $\det(\cU)\simeq\mathcal{O}$
 and it follows that $\cU \otimes \cU$ has a nowhere vanishing section.
 So does $\cH$ and finally we see that $\cU \otimes \cU \otimes \cH$ has a nowhere vanishing section.
 Consequently, its Euler class vanishes, hence the result (according to 
 point (1) in Par. \ref{num:Borel_classes}).
\end{proof}

\begin{rem}
There are also formulas encoding the associativity of the ternary tensor product
 of symplectic bundles. Such a formula can be expressed by considering five rank $2$
 symplectic bundles. We work over $P=(\HP^\infty)^5$ and write
 $\cU_i$ for the pullback of universal rank $2$ vector $\Sp$-bundle over the $i$-th copy
 of $P$. The formula then amounts to the equality of total Borel classes:
$$
b_t((\cU_1 \otimes \cU_2 \otimes \cU_3) \otimes \cU_4 \otimes \cU_5)
 =b_t(\cU_1 \otimes (\cU_2 \otimes \cU_3 \otimes \cU_4) \otimes \cU_5).
$$
Each of these Borel classes can be computed in terms of the total ternary law $F_t(x,y,z)$
 using the symplectic splitting principle (Remark \ref{num:sp-splitting-principle})
 and introducing three Borel roots for each of the rank $8$ $\Sp$-bundles
 $\cU_1 \otimes \cU_2 \otimes \cU_3$ and $\cU_2 \otimes \cU_3 \otimes \cU_4$.
 We will not use these formulas in the sequel, but we will come back to this in further work.
\end{rem}

\begin{num}\textit{The oriented case}. Suppose that $\E$ is $\GL$-oriented,
 and consider the induced $\Sp$-orientation (Rem. \ref{rem:link_orientations}).
 In this case, one can express the total ternary law $F_t(x,y,z)$ in terms
 of the formal group law associated with the $\GL$-orientation.
 Indeed, the Borel classes can be expressed in terms of Chern classes:
\begin{equation}\label{eq:Borel&Chern}
b_i(V,\psi)=(-1)^i.c_{2i}(V).
\end{equation}
This is easily seen for a symplectic bundle of rank $2$, using the fact that its first Chern class is trivial.
 The result for symplectic bundles of higher rank follows from the splitting principle.
 So, using the notations of the above definition, one can use the following equality:
\begin{equation}\label{eq:oriented_ternary}
c_t(U_1 \otimes U_2 \otimes U_3)=\sum_{i=2l} (-1)^l.F_l(x,y,z).t^i
\end{equation}
where $\cU_i=(U_i,\psi_i)$, $x=b_1(\cU_i)=-c_2(U_i)$. The identity \eqref{eq:oriented_ternary} becomes:
\begin{align*}
&\sum_{i=2l} (-1)^l.F_l(x,y,z).t^i=\prod_{e_1,e_2,e_3 \in \{\pm 1\}} F(F(e_1u_1,e_2u_2),e_3u_3), \\
&\quad x=u_1^2, y=u_2^2, z=u_3^2.
\end{align*}
We are now reduced to a formal computation. Let us state now the easy case of an additive
 formal group law. This will be used in the sequel and gives our first
 example.
\end{num}
\begin{prop}\label{prop:oriented_ternary}
Consider the notations above and assume $F(x,y)=x+y$.
 Then, using notations \eqref{eq:symmetric_mon},
 the ternary laws associated with the $\Sp$-oriented ring spectrum $\E$ are:
$$
F_l(x,y,z)=\begin{cases}
4.\sigma(x), & l=1, \\
6.\sigma(x^2)+4.\sigma(xy), & l=2, \\
4.\sigma(x^3)-4.\sigma(x^2y)+40.xyz, & l=3, \\
\sigma(x^4)-4.\sigma(x^3y)+6.\sigma(x^2y^2)+4.\sigma(x^2yz), & l=4.
\end{cases}
$$
\end{prop}

%

\section{Ternary laws associated with Chow-Witt groups} \label{sec:ternary}


\subsection{General principles}

\begin{num}
The aim of this section is to compute the ternary laws (Def. \ref{df:ternary})
 associated with Chow-Witt groups, or equivalently Milnor-Witt motivic cohomology.
 So, in the entire section we will work with the ring of coefficients $A=\E^{**}$ of
 some $\SL$-oriented ring spectrum $\E$. The ternary laws live in the ring
 of formal power series $A[[x,y,z]]$.

One can already obtain the following simple lemma which follows from
 the formula for the degree of the coefficients of the ternary laws
 \eqref{eq:aijkl_deg} and Proposition \ref{prop:ternary_monomials}.
\end{num}
\begin{lm}\label{lm:vanish_coef_ternary}
Let $\E$ be an $\SL$-oriented theory over a scheme $X$,
 with coefficients ring $\E^{**}=\E^{**}(X)$. We assume
 $\E^{n,i}=0$ in degree $m.(4,2)$ for $m\neq0$.

Then the ternary laws associated with $\E$ have the following form:
$$
F_l(x,y,z)=
\begin{cases}
2h.\sigma(x), & l=1, \\
2(h-\epsilon).\sigma(x^2)+a_{110}^2.\sigma(xy), & l=2, \\
2h.\sigma(x^3)+a^3_{210}.\sigma(x^2y)+a^3_{111}.xyz, & l=3, \\
\sigma(x^4)+a^4_{310}.\sigma(x^3y)+a^4_{220}.\sigma(x^2y^2)+a^4_{211}.\sigma(x^2yz), & l=4.
\end{cases}
$$
using convention \eqref{eq:symmetric_mon} for $\sigma$.
\end{lm}

\subsection{The case of Witt groups}\label{sec:Witt}

\begin{num}\label{num:initial_form_ternaryW}
We will now determine the ternary laws $F_l^W$ in the case of Witt unramified cohomology, 
 represented by $\HH \HW$.
 The underlying category of schemes $\Sch$ is that of $k$-schemes for a fixed base
 field $k$ of characteristic different from $2$ (see also Notations page \pageref{conventions}).

We can apply Lemma \ref{lm:vanish_coef_ternary} to this ring spectrum.
 Recall moreover that on $\HH \HW$, one gets $\epsilon=1$ and $h=0$. Let us restate
 the lemma in our particular case toi clarify the situation:
$$
F^W_l(x,y,z)=
\begin{cases}
0, & l=1, \\
-2.\sigma(x^2)+a_{110}^2.\sigma(xy), & l=2, \\
a^3_{210}.\sigma(x^2y)+a^3_{111}.xyz, & l=3, \\
\sigma(x^4)+a^4_{310}.\sigma(x^3y)+a^4_{220}.\sigma(x^2y^2)+a^4_{211}.\sigma(x^2yz), & l=4.
\end{cases}
$$
\end{num}

Our first step amounts consists in computing $F_l^W(x,y,0)$.
We use in particular the following proposition, which is a combination of results of Levine and Ananyevskiy.
\begin{prop}
Assume $\mathrm{char}(k) \neq 2$.
 Let $P=(\HP^\infty)^3$ and
 $\cU_i=(U_i,\psi_i)$ be the tautological rank $2$ symplectic bundle on the $i$-th coordinate of $P$
 (as in \ref{num:df:ternary}). Put $x=b_1(\cU_1)$, $y=b_1(\cU_2)$. 
 One can compute the following Pontryagin classes (Def. \ref{df:pontryagin})
 associated with $\HH \HW$:
$$
p_l^W(U_1 \otimes U_2)=\begin{cases}
-2.(x^2+y^2), & l=2, \\
0, & l=3, \\
(x^4+y^4)-2.x^2y^2, & l=4.
\end{cases}
$$
\end{prop}
For the computation of the odd classes, see \cite[7.9]{Anan19};
 for the even ones, see \cite[Prop. 9.1]{LevineEuler}.
 According to our definition, we get 
$$p_l^W(\cU_1 \otimes \cU_2)=F_l(x,y,0).$$
So the above proposition already gives us the following relations (corresponding to $l=2,3,4$):
\begin{align}
\label{eq:coef_ternary_W1}&a^2_{110}=0, \\
\label{eq:coef_ternary_W2}&a^3_{210}=0,  \\
\label{eq:coef_ternary_W3}& a^4_{310}=0, \ a^4_{220}=-2.
\end{align}

To get the remaining coefficients,
 we use the symmetric product $\Sym^*$ of symplectic vector bundles.
 In fact, the next proposition allows to determine $F_l^W(x,x,x)$
 and this will allow us to conclude.
\begin{prop}\label{prop:sym_power_ternaryW}
Assume that $6 \in k^\times$. \\
Let $\cU=(U,\psi)$ be the tautological rank $2$ symplectic bundle on $\HP^\infty$
 and $x$ be its first Borel class.
 Then the following computations hold in $H^*(\HP^\infty,\HW)$:
\begin{enumerate}
\item $b_i^W(\Sym^3 \cU)=\begin{cases}
(-3+\langle 3\rangle).x, & i=1,\\
(-4+\langle 3\rangle).x^2, & i=2, \\
0, & i>2.
\end{cases}$
\item $b_i^W(\cU^{\otimes,3})=\begin{cases}
0, & i=1,\\
-6.x^2, & i=2, \\
-8.x^3, & i=3, \\
-3.x^4, & i=4.
\end{cases}$
\end{enumerate}
\end{prop}
\begin{proof}
Our main tool will be the following computation of symplectic bundles \eqref{eqn:decomposition}:\footnote{This is
 the very point where the assumption on $k$ is needed.}
\begin{equation}\label{eq:sym_power_ternaryW1}
(U^{\otimes 3},\varphi^{\otimes 3})\simeq (U,\langle 2\rangle\varphi)\perp (U,\langle 6\rangle\varphi) \perp (\mathrm{Sym}^3U,\psi).
\end{equation}
Note that applying point (2) of Proposition \ref{prop:units&Borel},
 we get: 
$$
b^W_1(U,\langle 2\rangle\varphi)=\langle 2\rangle.x,
 \ b^W_1(U,\langle 6\rangle\varphi)=\langle 6\rangle.x.
$$
We start with the proof of (1).
We first do the computation of the left-hand column.
 The symplectic bundle $\Sym^3\cU$ is of rank $4$,
 so that we already get the required vanishing (by definition, see \eqref{eq:borel_pb_formula}).
 We apply the symplectic splitting principle
 to that bundle (Paragraph \ref{num:sp-splitting-principle}): in particular, after pullback along an affine morphism
 $p:X' \rightarrow \HP^\infty$, $\Sym^3\cU$ admits a splitting as two rank $2$ symplectic bundles,
 whose Euler classes we denote respectively $\alpha$ and $\beta$.
 Pulling back the decomposition \eqref{eq:sym_power_ternaryW1} to $X'$ yields a decomposition of $\cU^{\otimes,3}$ into a sum
 of four rank $2$ symplectic bundles and applying the Whitney sum formula of Par. \ref{num:Borel_classes},
 we get:
\begin{equation}\label{eq:sym_power_ternaryW2}
b^W_t(U^{\otimes 3},\varphi^{\otimes 3})=(1+\langle 2\rangle xt)(1+\langle 6\rangle xt)(1+\alpha t)(1+\beta t).
\end{equation}
As we already know $F_1^W$ and $F_2^W$ (Paragraph \ref{num:initial_form_ternaryW} and relation \eqref{eq:coef_ternary_W1}),
 we obtain the following equation by computing the coefficients of $t$ and $t^2$:
\begin{align*}
\alpha+\beta+(\langle 2\rangle+\langle 6\rangle).x & = 0 \\
\alpha\beta+(\alpha+\beta)(\langle 2\rangle + \langle 6\rangle).x+\langle 3\rangle.x^2 &=-6.x^2.
\end{align*}
Using the relation satisfied by the Borel roots $\alpha$ and $\beta$,
 we deduce that:
$$
b^W_1(\Sym^3 \cU)=(\langle -2\rangle+\langle -6\rangle).x,
 \ b_2^W(\Sym^3 \cU)=(-4+\langle 3\rangle).x^2.
$$
To conclude the proof of (1), it therefore suffices to show the following equality in $W(k)$:
\[
\langle -2\rangle+\langle -6\rangle  =  -3+ \langle 3\rangle.
\]
We need only to prove it either for a finite field (if $k$ is of positive characteristic), either for $\QQ$ in case $k$ is of characteristic zero.
 The first case is obvious as both forms have the same rank and same discriminant, the second case is obtained via a comparison of residues.

To get (2), it suffices now to finish the computation
 of \eqref{eq:sym_power_ternaryW2}. But using the computation above, we obtain:
$$
b^W_t(U^{\otimes 3},\varphi^{\otimes 3})=(1+\langle 2\rangle xt)(1+\langle 6\rangle xt)(1+(-3+\langle 3\rangle) xt+(-4+\langle 3\rangle)x^2t^2).
$$
Then, an easy computation allows us to conclude.
\end{proof}

\begin{num}
The case $l=3$ and $l=4$ of point (2) in the preceding computation together with relations \eqref{eq:coef_ternary_W2} and \eqref{eq:coef_ternary_W3}
 yield:
\begin{align}
& a^3_{111}=-8, \\
& 3-6+a^4_{211}=-3 \Rightarrow a^4_{211}=0.
\end{align}
Let us write the final result for the total ternary law associated with $\HH\HW$
\begin{equation}\label{eq:ternary_W}
F_t^W(x,y,z)=-2\sigma(x^2).t^2-8.xyz.t^3+\lbrack\sigma(x^4)-2\sigma(x^2y^2)\rbrack.t^4.
\end{equation}
\end{num}

\subsection{Final case}

We will now assemble our knowledge of ternary laws in Chow groups
 (Proposition \ref{prop:oriented_ternary}), represented over $k$ by
 the unramified Milnor K-theory $\HH\KM$,
 and Witt-cohomology from the preceding section.
 Indeed, recall there is canonical map of sheaves:
$$
\varphi:\KMW \rightarrow \KM \oplus \HW.
$$
We use the following lemma:
\begin{lm}
Let $P=(\HP^\infty_k)^3$. Then the morphism induced by $\varphi$ on cohomology:
$$
\varphi_*:\wCH^*(P) \rightarrow \CH^*(P) \oplus H^*(P,\HW)
$$
is injective.
\end{lm}
The proof immediately follows from the symplectic projective bundle formula \ref{thm:sp_pb_thm}
 and the fact $\varphi$ is compatible with the $\Sp$-orientations on each ring spectra.

This lemma allows us to combine the computations obtained for
 $\HH \KM$ and $\HH \HW$ respectively.
\begin{thm}\label{thm:HMW_ternary}
Let $k$ be a field such that $6 \in k^\times$.
 Then the ternary laws associated with 
 the $\Sp$-oriented ring spectra $\HH \KMW$ (Chow-Witt groups)
 or with $\HMW$ (Milnor-Witt motivic cohomology), over $k$, are:
$$
F_l(x,y,z)=
\begin{cases}
2h.\sigma(x), & l=1, \\
2(h-\epsilon).\sigma(x^2)+2h.\sigma(xy), & l=2, \\
2h.\sigma(x^3)+2h.\sigma(x^2y)+8(2h-\epsilon).xyz, & l=3, \\
\sigma(x^4)-2h.\sigma(x^3y)+2(h-\epsilon).\sigma(x^2y^2)+2h.\sigma(x^2yz), & l=4,
\end{cases}
$$
using the notations of \eqref{eq:symmetric_mon} and Remark \ref{rem:notations_GW}.
\end{thm}

Thus, by analogy with the classical oriented case, we introduce the following definition.
\begin{df}
Let $A$ be the ring of endomorphisms of the sphere spectrum in $\SH(\ZZ)$.
 We define the \emph{additive ternary laws} as the power series $F_l(x,y,z), l=1,2,3,4$
 with coefficients in $A$ defined by the formulas of the above theorem.

We will say that an $\Sp$-oriented ring spectra $\E$ has the additive ternary laws
 if for any scheme $S$ in $\Sch$, the associated ternary laws on $\E_X$
 are the additive ternary laws through the canonical map
 $A \rightarrow \E^{**}(X)$.
\end{df}

\begin{rem}
The computations of the previous section for the unramified Witt ring spectrum $\HW$ over $k$
 shows that $\HW$ has the additive ternary laws. Note however that $h=0$ and $\epsilon=-1$
 in $W(k)$ so that the formula simplifies to \eqref{eq:ternary_W}.
\end{rem}

With rational coefficients, the preceding theorem can be generalized.
\begin{cor}\label{cor:HMW_ternary}
The $\Sp$-oriented ring spectrum $\HMW\QQ$ over $\ZZ$ has the additive ternary laws.
\end{cor}
\begin{proof}
One uses the decomposition:
$$
\HMW\QQ \simeq \un_\QQ \simeq \un_{\QQ+} \oplus \un_{\QQ-}
 \simeq \HM\QQ \oplus \nu_*\HH \uW_\QQ
$$
where $\nu:\Spec{\ZZ} \rightarrow \Spec{\QQ}$ is the canonical open immersion
 (see \cite[Cor. 6.2]{DFJK1}).
\end{proof}

\section{Symplectic operations}\label{sec:operations}


\subsection{Unstable and additive operations}

\begin{num}
Let $S$ be any $\ZZp$-scheme.
 Using a method of Morel and Voevodsky, Panin and Walter proved in \cite[8.5]{PW1}
 that there exists a canonical weak $\AA^1$-homotopy equivalence:
\begin{equation}
HGr_S \rightarrow B\Sp_S.
\end{equation}
Moreover, when $S$ is regular, they also prove that one gets an isomorphism of abelian groups
 (we refer the reader to \cite[Th. 1.3]{ST}):
\begin{equation}\label{eq:htp_rep_sp_Kth}
KSp_0(S) \rightarrow [X,\ZZ \times B\Sp_S],
\end{equation}
where $[-,-]$ denotes the set of morphisms in the unpointed $\AA^1$-homotopy category over $S$.
 Concretely, to a symplectic vector bundle $\cV$ over $X$ of rank $2r$
 the above isomorphism associates the couple $(\gamma_\cV,r)$ where $\gamma_\cV:S \rightarrow BSp_S$
 is the map classifying $\cV$ and $2r$ is the rank of $\cV$.

In this section,
 we will use Riou's method to classify the following "operations".
\end{num}
\begin{df}
Let $\E$ be a spectrum over a regular scheme $S$.
 Let $(n,i)$ be a couple of integers.
 An \emph{$\Sp$-operation} (resp.\emph{ additive $\Sp$-operation})
 $\Theta$ of degree $(n,i)$ on $\E$ 
 will be a morphism of presheaves of sets (resp. abelian groups) on $\Sm_S$:
$$
\Theta:KSp_0 \rightarrow \E^{n,i}.
$$
We denote the set (resp. abelian group) of such operations 
 by $\Hom_{\mathrm{Sets}}(KSp_0,\E^{n,i})$
 (resp. $\Hom_{\mathrm{Ab}}(KSp_0,\E^{n,i})$)
\end{df}

\begin{num}
Let us now consider an $\Sp$-oriented ring spectrum $\E$
 over a regular scheme $S$. Let us put $\E^{**}=\E^{**}(S)$,
 as bigraded ring.

Recall that, according to \cite{PW1}\footnote{This is a corollary of the
 symplectic projective bundle formula \ref{eq:borel_pb_formula},
 totally analogous to the case of $\GL$-oriented theories.},
 one can compute the $\E$-cohomology of symplectic Grassmanians as:
\begin{equation}\label{eq:Sp-coh_HGr}
\E^{n,i}(HGr_S)=\big(E^{**}[[b_r,r \geq 1]]\big)^{(n,i)},
\end{equation}
where the exponent on the right-hand side denotes the subgroup of elements of
 degree $(n,i)$ with the convention that $b_r$ is of degree $(4r,2r)$.
 Explicitly, an element of the right hand side is a formal power series
 of the form:
\begin{equation}\label{eq:coh_class_HGr}
F=\sum_\alpha \left(a_\alpha.{\textstyle \prod}_{i \in \NN^*} b_i^{\alpha(i)}\right)
\end{equation}
where $\alpha$ runs over the applications $\alpha:\NN^* \rightarrow \NN$
 with finite support and $a_\alpha$ is an element of $\E^{**}$ of
 degree $(n-4.|\alpha|,i-2.|\alpha|)$.

Using the method of \cite{Riou}, one deduces a complete description
 of the $\Sp$-operations on $\E$.
\end{num}
\begin{thm}\label{thm:set_symplectic_op}
Consider the above notations.
 Then the canonical map:
$$
\Hom_{\Htp(S)}\big(\ZZ \times HGr,\Omega^{\infty}(\E(i)[n])\big)
 \rightarrow \Hom_{\mathrm{Sets}}(KSp_0,\E^{n,i}), \phi \mapsto \phi_*
$$
is bijective. Moreover, via the identification of the left hand-side with
 the set $\E^{n,i}(HGr_S)^\ZZ$,
 a sequence of formal power series $(F_r)_{r \in \ZZ}$
 of the form \eqref{eq:coh_class_HGr} corresponds to the operation
 which sends a symplectic bundle $\cV$ of rank $2r$
 over a smooth $S$-scheme $X$ to the (well-defined\footnote{Use the fact $b_i(\cV)$
 is nilpotent;}) element of $\E^{n,i}(X)$:
$$
\Theta_{F_*}([\cV]):=F_r\big(b_1(\cV),\hdots,b_r(\cV),0,\hdots\big).
$$
where $b_i(\cV)$ denotes the Borel classes of $\cV$ (Par. \ref{num:Borel_classes}).
\end{thm}
Note moreover that the map corresponding to $\Theta_{F_*}$ is pointed
 (for the obvious base points) if and only if $F_0=0$.
\begin{proof}
This is a consequence of \cite[1.2.9]{Riou} applied to the inductive
 system of smooth schemes $HGr_\bullet=(HGr_{n,d})_{n,d \in \NN}$ and to the
 $H$-group $E=\Omega^\infty(\E(i)[n])$. Indeed, formula \eqref{eq:Sp-coh_HGr}
 implies that "$HGr_\bullet$ does not unveil phantoms in $E$" in the sense
 of \cite[1.2.2]{Riou}. Moreover, "$\pi_0 HGr$ is generated by $HGr$"
 in the sense of \cite[1.2.5]{Riou} by application of \cite[8.1]{PW1}.
\end{proof}


\begin{num}
According to \cite[Lem. 6.2.2.2]{Riou},
 for any commutative ring $A$ and any integer $n>0$,
 there exists a unique family of group morphisms of the form
$$
\psi^A_n:(1+t.A[[t]],\times) \rightarrow (A,+)
$$
which are natural in A, such that $\psi^A_n(1+a.t)=a^n$
 and which vanishes on $1+t^{n+1}.A[[t]]$.
By analogy with the oriented case, 
 one defines an additive $\Sp$-operation
 $\tilde \chi_{2n}^\E$
 of degree $(4n,2n)$ on $\E$. Given any symplectic bundle 
 $\cV$ on a smooth $S$-scheme $X$, one sets:
\begin{equation}
\tilde \chi_{2n}^\E([\cV])=\psi_n^{\E^{**}}(b_t(\cV)),
\end{equation}
where $b_t(\cV)$ is the total Borel class of $\cV$,
 Formula \eqref{eq:Borel_total}. (The latter formula shows that
 $\tilde \chi_{2n}^\E$ is indeed additive.)
 We also put $\tilde \chi^E_0=1$.
Note in particular that when $\cV=\cU$ has rank $2$, one gets by construction
 for $n \geq 0$
\begin{equation}\label{eq:carac_chi_n}
\tilde \chi_{2n}^\E([\cU])=b_1(\cU)^n.
\end{equation}

Recall from Panin-Walter's symplectic projective bundle theorem
 (Th. \ref{thm:sp_pb_thm})
 that one obtains an isomorphism of bigraded $\E^{**}$-modules:
$$
\E^{**}[[b]] \rightarrow \E^{**}(\HP^\infty_S),
 b \mapsto b_1(\cU),
$$
where the indeterminate $b$ has degree $(4,2)$.
 As in \cite[Prop. 6.2.2.1]{Riou}, one can compute all
 the additive $\Sp$-operations on $\E$.
\end{num}
\begin{thm}\label{thm:add_op}
Consider the above notations.
 Let $\cU$ be the tautological symplectic bundle of rank $2$
 on $\HP^\infty_S$.
 Then the canonical morphism of graded $\E^{**}$-modules
$$
\Hom_{\mathrm{Ab}}(KSp_0,\E^{**})
 \rightarrow \E^{**}(\HP^\infty_S) \simeq \E^{**}[[b]],
 \phi \mapsto \phi_{\HP^\infty_S}([\cU])
$$
is an isomorphism. Moreover, for $n \geq 0$,
 the additive $\Sp$-operation 
 $\tilde \chi_{2n}^\E$ defined above is sent to $b^n$ via this isomorphism. 
\end{thm}
\begin{proof}
The injectivity of the map follows from the symplectic bundle
 principle (Par. \ref{num:sp-splitting-principle})
 and the universal property of $\HP^\infty \simeq B\Sp_2$.
 The last assertion follows from relation \eqref{eq:carac_chi_n}.
 Thus the surjectivity of the map follows from the existence
 of the operations $\tilde \chi_{2n}^{\E}$.
\end{proof}

\begin{num}
The lemma of Riou used in the above proof is a smart way of constructing
 $\tilde \chi_{2n}$. A more classical way is to use the symplectic splitting
 principle \ref{num:sp-splitting-principle}.
 Indeed, $\tilde \chi_{2n}(\cV)$ is uniquely defined in terms of the Borel roots
 $\xi_i$ of $\cV$ by the property:
$$
\tilde \chi_{2n}(\cV)=\sum_i \xi_i^n.
$$
Thus one can express $\tilde \chi_{2n}(\cV)$ in terms of the Borel classes of $\cV$
 using the fact that $b_i=b_i(\cV)$ is the $i$-th elementary symmetric function in the $\xi_i$
 and using the classical expression of the symmetric power sum polynomials in
 terms of the elementary symmetric functions. 
 For example, a classical formula in terms of determinant is:
\begin{equation}\label{eq:determinant_tilde_chi}
\begin{split}
\tilde \chi_{2n}(\cV)=
\left|
\raisebox{0.5\depth}{
\xymatrix@=0.1ex{
b_1 & 1\ar@{-}[rrrddd] & && \ar@{}|/-30pt/0[lllldddd]\\
2b_2\ar@{.}[ddd] & b_1\ar@{-}[rrrddd] && & \\
& b_2\ar@{-}[rrdd]\ar@{.}[dd] & & & \\
&  & & & 1 \\
nb_n & b_{n-1}\ar@{.}[rr] & & b_2 & b_1
}
}
\right|.
\end{split}
\end{equation}
A more useful formula is given by Newton's relations:
\begin{equation}\label{eq:Newton_ident}
\tilde \chi_{2n}-b_1\tilde \chi_{2n-2}+b_2\tilde \chi_{2n-4}
+ \hdots +(-1)^{n-1}b_{n-1}\tilde \chi_2+(-1)^nnb_n=0.
\end{equation}
\end{num}

\begin{ex}\label{ex:add_op_HMW}
Let us assume that either $k$ is a perfect field $k$ and $R=\ZZ$
 or $k=\ZZp$ and $R=\QQ$. Given any $k$-scheme $X$,
 we define the (Eilenberg-MacLane) Milnor-Witt $(n,i)$-space over $S$ by the formula 
$$
K(\tilde R_S(i),n)=\Omega^\infty\big(\HMW R_S(i)[n]\big).
$$
It is the analogue of the $(n,i)$-th Eilenberg-MacLane motivic space $K(R(i),n)$
 (see \cite[\textsection 6.1]{V_ICM} for $(n,i)=(2n,n)$).
According to the absolute purity theorem for Milnor-Witt motivic cohomology
 (see \cite{DFJK1})), this space is stable under arbitrary pullback between regular
 $k$-schemes.

The previous theorem applied to the $\Sp$-oriented ring spectrum $\HMW R_S$
 gives a canonical additive $\Sp$-operation of degree $(2n,n)$:
$$
\tilde \chi^R_{2n}:KSp_0 \rightarrow H_{MW}^{4n,2n}(-,R) \simeq \wCH^{2n}_R,
$$
the last isomorphism being \eqref{eq:MWZ&wChow} and \eqref{eq:MWQ&wChow}.
It corresponds to a morphism of $H$-groups:
$$
\ZZ \times B\Sp_S \rightarrow K(\tilde R_S(2n),4n).
$$
These operations form a base of all the additive $\Sp$-operations over $S$
 with values in Chow-Witt groups $\wCH^n_R$.
\end{ex}

\subsection{Stable operations and hermitian K-theory}\label{subsec:stableop}

\begin{num}\label{num:desuspension}
We now study stable operations over a regular scheme $S$, still following the method of Riou.
 A technical difference between symplectic K-theory and usual K-theory
 is that the former is $(8,4)$-periodic while the latter is $(2,1)$ periodic.
 Therefore the natural sphere that comes into play is
 $H:=(\HP^1)^{\wedge,2} \simeq \un(4)[8]$.\footnote{This isomorphism follows from the fact $\HP^1=Q_4$ and \cite{ADF}.}

Given any spectrum $\E$ over $S$, we get a tautological stability
 isomorphism for any smooth $S$-scheme $X$:
$$
\sigma^\E:\E^{n-8,i-4}(X) \xrightarrow \sim \tilde \E^{n,i}(H \wedge X_+)
$$
where $\tilde \E$ is the associated reduced cohomology theory.
 When $\E$ admits a ring structure,
 this isomorphism can be expressed by the cup-product
 with a tautological class $\sigma_X \in \tilde \E^{8,4}(H)$.
 When $\E$ is in addition $\Sp$-oriented, this class is induced by 
 $b_1(\cU_1) \cdot b_1(\cU_2) \in \E^{8,4}(\HP^1 \times \HP^1)$
 where $\cU_i$ is the pullback of the tautological symplectic bundle 
 on the $i$-th factor.

Similarly, the multiplication map in symplectic K-theory
$$
\KSp_0(X)
 \xrightarrow{(\cU_i-\cH) \cdot (\cU_i-\cH)} \KSp_0(\HP^1 \times \HP^1 \times X)
$$
induces, according to the symplectic bundle theorem, an isomorphism:
$$
\sigma^\KSp:\KSp_0(X) \rightarrow\tilde \KSp_0(H \wedge X_+).
$$
Given an $\Sp$-operation $\theta$ on a spectrum $\E$
 of degree $(n,i)$,
 we define a new associated $\Sp$-operation $\omega_H(\theta)$
 on $\E$ of degree $(n-8,i-4)$ by the following commutative diagram:
\begin{equation}\label{eq:def_desuspension_op}
\begin{split}
\xymatrix@R=14pt@C=44pt{
\KSp_0(X)\ar^-{\sigma^{\KSp}}_-\sim[r]\ar_{\omega_H(\theta)}[d]
 & \tilde \KSp_0(H \wedge X_+)\ar^\theta[d] \\
\E^{n-8,i-4}(X)\ar^-{\sigma^{\E}}_-\sim[r] & \tilde \E^{n,i}(H \wedge X_+).
}
\end{split}
\end{equation}
In particular, we get two projective systems indexed by integers $r \geq 0$:
\begin{align*}
\hdots \Hom_{\mathrm{Set}}\big(KSp_0,\E^{n+8r,i+4r}\big)
 \xleftarrow{\omega_H} \Hom_{\mathrm{Set}}\big(KSp_0,\E^{n+8r+8,i+4r+4}\big)
 \hdots \\
\hdots \Hom_{\mathrm{Ab}}\big(KSp_0,\E^{n+8r,i+4r}\big)
 \xleftarrow{\omega_H} \Hom_{\mathrm{Ab}}\big(KSp_0,\E^{n+8r+8,i+4r+4}\big)
 \hdots
\end{align*}
Their projective limits agree as stable operations must be additive
 (see \cite[3.5]{Vish}).
\end{num}
\begin{df}\label{df:stable_op}
Consider the above notations.
We define the abelian group of stable $\Sp$-operations on $\E$
 of degree $(n,i)$ as the projective limit
 of one of the two projective system above:
\begin{align*}
\Hom_{\mathrm{St}}\big(KSp_0,\E^{n,i}\big)
 &=\lim_{r \in \NN} \Hom_{\mathrm{Sets}}\big(KSp_0,\E^{n+8r,i+4r}\big) \\
 &\simeq \lim_{r \in \NN} \Hom_{\mathrm{Ab}}\big(KSp_0,\E^{n+8r,i+4r}\big).
\end{align*}
In particular, such an operation  is a sequence
 $(\Theta_r)_{r \in \NN}$ of additive $\Sp$-operations
 $\Theta_r:\KSp \rightarrow \E^{n+8r,i+4r}$ such that for any $r \geq 0$,
 $\Theta_r=\omega_H(\Theta_{r+1})$.
\end{df}

\begin{num}
Recall from \cite[6.1]{Riou0} that a \emph{naive $H$-spectrum} over a scheme $S$
 is the datum of a sequence $(E_n,\sigma_n)_{n \in \NN}$
 such that $E_n$ is an object of $\Htpp(S)$ and $\sigma_n:H \wedge E_n \rightarrow E_n$
 is a map in $\Htpp(S)$ whose adjoint map $E_n \rightarrow \Omega_H E_n$
 is an isomorphism.
 Every spectrum $\E$ if $\SH(S)$ determines a naive $H$-spectrum
 whose $n$-th term is $E_n=\Omega^\infty \Sigma_H^n \E$.
 Reciprocally, any naive $H$-spectrum $(E_n,\sigma_n)$ admits a lifting
 to an object of $\SH(S)$, and the lifting is unique provided
$$
\derR^1 \lim_{n \in \NN} \Hom_{\Htpp(S)}(S^1 \wedge E_n,E_n)=0.
$$

An important example for us is provided by the naive $H$-spectrum over any $\ZZp$-scheme $S$
$$
(\ZZ \times B\Sp,\ZZ \times B\Sp,\hdots)
$$
whose transition maps are all equal to the multiplication by the element
 $\sigma^2 \in [H,\ZZ\times B\Sp]_*$ corresponding to the element
 $[\cU_1-\cH].[\cU_2-\cH]$ in $\KSp_0(\HP^1 \times \HP^1)$,
 where $\cU_i$ is the tautological symplectic bundle on the $i$-th factor of $\HP^1 \times \HP^1$.
 As explained above, this naive $H$-spectrum lifts as an object of $\SH(S)$.
 Over $S=\Spec{\ZZp}$, the ambiguity in this lifting vanishes according to \cite[13.2]{PW1},
 thus giving a canonical spectrum $\KSp_{\ZZp}$. Over an arbitrary $\ZZp$-scheme $S$,
 we define $\KSp_S$ by pullback. According to \cite{ST}, one gets a canonical isomorphism in $\SH(S)$:
\begin{equation}\label{eq:KSp&GW}
\KSp_S \simeq \KQ_S(2)[4].
\end{equation}
Let us recall the following proposition from \cite[Lem. 6.4]{Riou0}.
\end{num}
\begin{prop}
 Given any spectrum $\E$ over $S$, there is a short exact sequence:
\begin{align*}
0 \rightarrow &\derR^1 \lim_{r \in \NN} \Hom_{\mathrm{Ab}}(KSp_0,\E^{n+8r-1,i+4r})
 & \longrightarrow \Hom_{\SH(S)}\big(\KSp_S,\E(i)[n]\big) \\
 & \longrightarrow \Hom_{\mathrm{St}}\big(KSp_0,\E^{n,i}\big)
 \rightarrow 0,
\end{align*}
where the transition maps in the left-hand side projective system are given by the
 desuspension maps $\omega_H$ (Paragraph \ref{num:desuspension}).
\end{prop}

\begin{num}\label{num:HMW_vanishing}
Let $k$ be a perfect field of characteristic different from $2$. According to \cite[4.1.2]{DF1},
 we get the following vanishing of the Milnor-Witt motivic cohomology groups of $k$
 with integral coefficients:
$$
H^{n,m}_{MW}(k,\ZZ)=0 \text{ if } n>m \text{ or } (m<0 \text{ and } n \neq m).
$$
Regarding rational motivic cohomology, we get from \cite{DFJK1} that:
$$
H^{n,m}_{MW}(\ZZp,\QQ)=K_{2m-n}^{(m)}(\ZZp) \otimes \QQ \oplus H^{n-m}_\zar(\QQ,\HW \otimes \QQ)
$$
where $\HW$ is the unramified Witt-sheaf over $\Sm_\QQ$. Given Borel's computations of
 the K-theory of integers, these groups vanish in the same range as in the previous
 case.\footnote{\label{fn:Borel} Recall that from the computations
 of \cite[Prop. 12.2]{Borel}, one classically derives the following ones for rational motivic cohomology:
\begin{equation}\label{eq:Borel}
H^{n,m}_M(\ZZ,\QQ)=K_{m-2n}^{(m)}(\ZZ)_\QQ=
\begin{cases}
\QQ & (n,m)=(0,0), (1,2r+1), r>0 \\
0 & \text{otherwise.}
\end{cases}
\end{equation}
}
\end{num}
\begin{cor}\label{cor:pre_comput_stable-op}
Assume we are in one of the following cases:
\begin{itemize}
\item $S$ is the spectrum of a perfect field $k$
 of characteristic not $2$, and $R=\ZZ$;
\item $S$ is the spectrum of $\ZZp$, and $R=\QQ$.
\end{itemize}
Then for any integer $n$,
 the following canonical map, appearing in the previous proposition, is an isomorphism:
$$
\Hom_{\SH(S)}\Big(\KSp_S,\HMW R(2n)[4n]\Big) \xrightarrow \sim \Hom_{\mathrm{St}}\big(KSp_0,\wCH^{4n}_R\big).
$$
\end{cor}

\begin{num}\label{num:pre_computation_stable_op_HMW}
Indeed, in view of the vanishing recalled before the corollary,
 the projective system appearing on the left hand side of the short exact sequence in the previous section
 is just $0$ at each degree.
 Besides, one can give another expression for the projective system of the right hand-side
 appearing in Definition \ref{df:stable_op}. Using Theorem \ref{thm:add_op}, the above vanishing
 and the fact
$$
H^{0,0}_{MW}(S,R) \simeq GW(S)_R,
$$
we get that the two above groups are given by the projective limit of a tower, indexed by $r \geq 0$,
 of the form:
\begin{equation}\label{eq:pre_comput_stable-op}
\hdots \leftarrow \GW(S)_R.\tilde \chi^R_{2n+4r} \xleftarrow{\omega_H} \GW(S)_R.\tilde \chi^R_{2n+4r+4} \leftarrow \hdots
\end{equation}
where we have denoted by $\tilde \chi^R_{2n}:KSp_0 \rightarrow H_{MW,R}^{4n,2n}$
 the additive $\Sp$-operations of Example \ref{ex:add_op_HMW}. Note that
 the desuspension operators $\omega_H$ (Paragraph \ref{num:desuspension}) are $GW(S)$-linear.
 The next subsection is devoted to compute explicitly the transition maps in the above
 projective system.
\end{num}

\subsection{Stabilization for Milnor-Witt motivic cohomology}

\begin{num}\label{num:desuspension_add_op_HMW}
Consider the notations and assumptions of 
 Corollary \ref{cor:pre_comput_stable-op}
 and Paragraph \ref{num:pre_computation_stable_op_HMW}.
 According to formula \eqref{eq:pre_comput_stable-op}, we know a priori
 that for any $n \geq 0$,
 there exists an isomorphism class of quadratic form $\psi_{2n+4} \in \GW(S)_R$
 such that:
\begin{equation}\label{eq:relation-stability}
\omega_H\left(\tilde \chi_{2n+4}^R\right)=\psi_{2n+4}.\tilde \chi_{2n}^R.
\end{equation}
For normalization purposes, we put: $\psi_0=\psi_2=1$.
\end{num}
\begin{thm}\label{thm:desuspension_add_op_HMW}
Consider the above notations.
 We assume one of the following hypothesis:
\begin{enumerate}
\item[(a)] $S$ is the spectrum of a perfect field $k$ 
 such that $6 \in k^\times$, and $R=\ZZ$;
\item[(b)] $S$ is the spectrum of $\ZZp$, and $R=\QQ$.
\end{enumerate}
Then for any integer $n \geq 0$, the quadratic form appearing in
 relation \eqref{eq:relation-stability} is:
\begin{equation}\label{eq:denomin_Borel_char}
\psi_{2n+4}=\begin{cases}
\frac 12(2n+4)(2n+3)(2n+2)(2n+1).h & \text{if $n\geq 0$ is even,} \\
(2n+4)(2n+2).\big((2n^2+4n+1).h-\epsilon\big) & \text{if $n$ is odd.} \\
\end{cases}
\end{equation}
\end{thm}
\begin{proof}
 To determine an additive $\Sp$-operation over $S$,
 we know from Theorem \ref{thm:add_op} that we need only to apply it
 to the element $[\cU] \in KSp_0(\HP^\infty_S)$. Going back to the
 defining diagram \eqref{eq:def_desuspension_op} for $\omega_H$,
 and using the fact that $\tilde \chi_{2n}^R(\cU)=b_1(\cU)^n=u^n$,
 we get the following relation in the cohomology group
 $H_{MW}^{4n+8,2n+4}(\HP^1 \times \HP^1 \times \HP^\infty_S)$:
\begin{equation}\label{eq:desuspension_add_op_HMW}
\tilde \chi_{2n+4}^R\big((\cU_1-\cH)(\cU_2-\cH)\cU\big)=\psi_{2n+4}.u_1u_2u^n,
\end{equation}
where we denoted $u_1$ and $u_2$ for the first Borel class of the tautological symplectic bundle
 on the first and second coordinate of $\HP^1 \times \HP^1 \times \HP^\infty_S$.

Using the computation of the ternary laws for Milnor-Witt cohomology,
 it is possible to determine $\psi_{2n}$. However, it is possible to simplify substantially
 this computation by remembering that the class $\psi_{2n}$ in $GW(S)_R$ is determined
 by its rank and its class in the Witt ring $W(S)_R$. On the other hand, we have two canonical
 maps:
$$
\HMW R_S \rightarrow \HM R_S,\ \HMW R_S \rightarrow \HH \uW_{R,S}
$$
according to \cite[3.4.1]{DF2} and \cite[(proof of) 4.1.2]{DF1} under assumption (a)
 and \cite[Cor. 6.2]{DFJK1} under assumption (b).
 These maps induces respectively the rank and the projection map on the cohomology groups in degree $(0,0)$.
 Thus we need only to specialize our computations either to motivic cohomology
 or to unramified Witt cohomology. This will be done below in
 Propositions \ref{prop:desuspension_add_op_HM}
 and Corollary \ref{cor:desuspension_add_op_W}.
\end{proof}

\begin{num}\label{num:chi_tilda_HM}
We consider the hypothesis of Theorem \ref{thm:desuspension_add_op_HMW}.
 Recall from \cite[6.2.2.1]{Riou} that there are canonical operations:
$$
\chi_i:\ZZ \times B\GL \rightarrow K\big(R(i),2i\big)
$$
where $K(R(i),2i)=\Omega^\infty\big(\HM R_S(i)[2i]\big)$ is
 the motivic Eilenberg-MacLane space of degree $(2i,i)$.\footnote{\label{fn:RiouZ}
 The statements of \cite[Section 6.2]{Riou} are given over a perfect base field $k$.
 However, the proof applies equally to the case $S=\Spec(\ZZp)$ (or even $S=\Spec(\ZZ)$),
 $R=\QQ$,
 given the vanishing of rational motivic cohomology of $\Spec(\ZZ)$ due to Borel's computations
 (see footnote \ref{fn:Borel} page \pageref{fn:Borel}).}

On the other hand, we defined additive $\Sp$-operations 
 by applying Theorem \ref{thm:add_op} to the motivic cohomology ring
 spectrum $\HM R_S$:
$$
\tilde \chi_{2n}^M:\ZZ \times B\Sp \rightarrow K\big(R(2i),4i\big).
$$
\end{num}
\begin{lm}\label{lm:compare_operations_BGL}
Consider the above notations and assumptions. Let $f:B\Sp \rightarrow B\GL$
 be the canonical forgetful map. Then for any $n>0$ one has:
$$
2.\tilde \chi_{2n}^M=\chi_{2n} \circ f.
$$
\end{lm}
Note that by definition, $\tilde \chi_0=1=\chi_0$.
\begin{proof}
Let us denote by $b_i^M$ the Borel classes associated with the canonical
 $\Sp$-orientation of $\HM R_S$. We prove the result by induction on 
 $n\geq 1$.

For $n=1$, and a symplectic bundle $(U,\psi)$,
 we get from formulas \eqref{eq:determinant_tilde_chi}
 and \eqref{eq:Borel&Chern}, ${\tilde\chi}^M_2(U,\psi)=b^M_1(U,\psi)=-c_2(U)$.
 On the other hand, using \cite[6.2.2.3]{Riou}, we obtain
 $\chi_{2}=c_1^2(U)-2c_2(U)=-2c_2(U)$. These two equalities allow to conclude.
Then the induction step is provided by the following computation
 (where we suppress $f$ for readability):
\begin{align*}
\chi_{2n+2}  & =-\sum_{i=1}^{n} c_{2i}\chi_{2n-2i+2}-(2n+2)c_{2n+2} \\
 & = -2.\sum_{i=1}^{n} (-1)^ib^M_{i}{\tilde\chi}^M_{2n-2i+2}-(2n+2)(-1)^{n+1}b^M_{n+1} \\
 & = -2\left(\sum_{i=1}^{n} (-1)^ib^M_{i}{\tilde\chi}^M_{2n-2i+2}+(-1)^{n+1}(n+1)b^M_{n+1}\right)
 = 2{\tilde\chi}^M_{2n+2},
\end{align*}
where the first (resp. second, last) equality 
 follows from \cite[6.2.2.3]{Riou} (resp. \eqref{eq:Borel&Chern} and
 the induction hypothesis, \eqref{eq:Newton_ident}).
\end{proof}

\begin{prop}\label{prop:desuspension_add_op_HM}
For any $n>0$, we have:
$$
\omega_H({\tilde\chi}^M_{2n+4})=(2n+4)(2n+3)(2n+2)(2n+1){\tilde\chi}^M_{2n}.
$$
Consequently, $\rk(\psi_{2n+4})=(2n+4)(2n+3)(2n+2)(2n+1)$ for any $n\geq 0$.
\end{prop}
\begin{proof}
According to the plus part of formula \eqref{eq:pre_comput_stable-op},
 we have a priori:
$$
\omega_H({\tilde\chi}^M_{2n+4})= r_{2n+4}.{\tilde\chi}^M_{2n}
$$
where $r_{2n+4}$ is an element of $H^{0,0}_M(S,R)=R$.\footnote{In fact,
 $r_{2n+4}=\rk(\psi_{2n+4})$.}
We know from \cite[6.2.3.2]{Riou}
 that $\Omega_{\PP^1}(\chi_{2n+4})=(2n+4)\chi_{2n+3}$.
 Therefore, as $2$ is non-zero divisor in $R$,
 the proposition follows the previous lemma and the obvious fact:
 $\omega_H({\tilde\chi}^M_{2n+4} \circ f)
 =\Omega_{\PP^1}^4(\tilde\chi^M_{2n+4}) \circ f$.
 Beware the particular case $n=0$, as $\tilde \chi_0=\chi_0$.
\end{proof}

\begin{num}
The main point to prove the above theorem
 is to determine the Witt part $\bar \psi_{2n+4} \in \W(S)_R$ of 
 the quadratic form $\psi_{2n+4} \in \GW(S)_R$ of
 Paragraph \ref{num:desuspension_add_op_HMW}.
 So we consider the assumptions of this theorem
 and we let $\tilde \chi^W_{2n+4}$ (resp. $b_i^W$) be the
 $\Sp$-operation (resp. Borel class) associated with the $\Sp$-oriented
 ring spectrum $\HW_{R,S}$.
 Specializing relation \eqref{eq:desuspension_add_op_HMW}, we get for
 any $n \geq 0$:
$$
{\tilde \chi}_{2n+4}^W\big((\cU_1-\cH)(\cU_2-\cH)\cU\big)=\bar \psi_{2n+4}.u_1u_2u^n,
$$
where $u_1=b_1^W(\cU_1)$, $u_2=b_1^W(\cU_2)$, $u=b_1^W(\cU)$.
\end{num}
\begin{prop}
Under the above assumptions, we have:
\[
{\tilde{\chi}}^{W}_{2n+4}(\cU_1\cU_2\cU)=
\begin{cases}
 4u^{n+2} & \text{if $n$ is even,} \\
 -(2n+4)(2n+2)u_1u_2u^{n} & \text{if $n$ is odd.}
\end{cases}
\]
\end{prop}
\begin{proof}
Put $b_i^W=b_i^W(\cU_1\cU_2\cU)$ and
 ${\tilde{\chi}}^{W}_{2n+4}={\tilde{\chi}}^{W}_{2n+4}(\cU_1\cU_2\cU)$.
Let us first start by computing the Borel classes using the ternary law
 of unramified Witt theory, Formula \eqref{eq:ternary_W}
 and the relation $u_i^2=0$:
$$
b_i^W=
\begin{cases}
0 & i=1, i>4, \\
-2u^2 & i=2, \\
-8u_1u_2u & i=3, \\
u^4 & i=4.
\end{cases}
$$
We derive from this computation and the Newton's identity relation
 \eqref{eq:Newton_ident} the following relation for $n>2$:
\begin{equation}\label{eq:desuspension_add_op_W}
{\tilde{\chi}}^{W}_{2n+4}
 =2u^2.\tilde{\chi}^{W}_{2n}
-8u_1u_2u.\tilde{\chi}^{W}_{2n-2}
-4u^4.\tilde{\chi}^{W}_{2n-4}.
\end{equation}
On the other hand, one express the first $\Sp$-operations
 using again the first computation and Formula \eqref{eq:determinant_tilde_chi}:
$$
\tilde \chi^W_{2n+4}=
\begin{cases}
4u^2 & n=0, \\
-24u_1u_2u & n=1, \\
4u^4 & n=2.
\end{cases}
$$
Finally, one proves the lemma by induction on $n$
 using relation \eqref{eq:desuspension_add_op_W}.
\end{proof}

\begin{cor}\label{cor:desuspension_add_op_W}
For any
 $n\geq 0$, we have
$$
\omega_H(\tilde \chi^W_{2n+4})=
\begin{cases}
0 & \text{if $n$ is even,} \\
-(2n+4)(2n+2).\tilde \chi^W_{2n}  & \text{if $n$ is odd}.
\end{cases}
$$
In other words,
 the image of $\psi_{2n+4}$ in $W(S)_R$ is
 $0$ for $n$ even and $-(2n+4)(2n+2)$ if $n$ is odd.
\end{cor}

\subsection{Rational stable $\Sp$-operations}

\begin{num}
Consider again the notations of Paragraph \ref{num:desuspension_add_op_HMW}.
The next step is to provide conditions under which inverting the quadratic forms
 $\psi_{2n}$  in the Grothendieck-Witt ring $\GW(S)_R$ is sensible. For $n$ even,
 the forms are hyperbolic and inverting them would erase all quadratic information. Thus, we are led
 to consider the multiplicative system $S_\psi$ of $\GW(S)_R$ generated 
 by $\{\psi_{2n}\vert n \text{ odd}\}$.
 Note that 
$$
\GW(\ZZ)_\QQ \simeq \QQ.h \oplus \QQ.\epsilon \simeq \GW(\QQ)_\QQ
$$
so we restrict our attention to the case of a perfect field $k$.
\end{num}

\begin{lm}
Let $P$ be an ordering of $k$ and let $s_P:\mathrm{W}(k)\to \ZZ$ be the corresponding signature homomorphism, i.e. the homomorphism characterized by 
\[
s_P(\langle a\rangle)=\begin{cases} 1 & \text{ if $a$ is positive w.r.t. $P$,} \\
-1 & \text{ if $a$ is negative w.r.t. $P$.}\end{cases}
\]
Then $s_P(\psi_{2n+4})=-(2n+4)(2n+2)\neq 0$ for any odd $n\in \NN$.
\end{lm}

\begin{proof}
It suffices to observe that $1$ is a square, and then is positive w.r.t. $P$. It follows that $-1$ is negative and we conclude.
\end{proof}

\begin{prop}
Let $\mathcal P$ be the set of orderings of $k$. Then,
\[
S_{\psi}^{-1}\mathrm{W}(k)\simeq \bigoplus_{P\in\mathcal P}\QQ.
\]
\end{prop}

\begin{proof}
The above lemma shows that the map
\[
\mathrm{W}(k)\xrightarrow{\sum s_P} \bigoplus_{P\in\mathcal P} \ZZ
\]
induces a well-defined map as in the statement, i.e. we have a commutative diagram
\[
\xymatrix@C=35pt@R=14pt{
\mathrm{W}(k)\ar[r]^-{\sum s_P}\ar[d] & \bigoplus_{P\in\mathcal P} \ZZ\ar[d] \\
S_{\psi}^{-1}\mathrm{W}(k)\ar[r] & \bigoplus_{P\in\mathcal P}\QQ.
}
\]
Besides, the kernel and cokernel of the top homomorphism are $2$-primary torsion.
 It follows immediately that the bottom map is surjective. Let now $y$ be in the kernel of 
\[
S_{\psi}^{-1}\mathrm{W}(k)\xrightarrow{\sum s_P} \bigoplus_{P\in\mathcal P}\QQ,
\]
We may write $y=\frac xs$ with $s\in S_{\psi}$, $x\in W(k)$. One deduces:
$$\sum s_P(x)(\sum s_P(s))^{-1}=0$$
thus showing that $\sum s_P(x)=0$. It follows that $2^rx=0$ for some $r\in\NN$. Now, we have $8\vert (2n+4)(2n+2)$ in $\mathrm{W}(k)$ and it follows that for any $r\in\NN$ there exists an odd $n$ such that $2^r\vert (\psi_{2n+4}\cdot \psi_{2n}\cdot \ldots\psi_2)$. The map is thus injective.
\end{proof}

\begin{cor}
The signature and rank homomorphisms induce an isomorphism
\[
S_{\psi}^{-1}\mathrm{GW}(k)\simeq \QQ\oplus \bigoplus_{P\in\mathcal P}\QQ\simeq \mathrm{GW}(k)\otimes \QQ
\]
\end{cor}
\begin{proof}
We have an exact sequence of $\GW(k)$-modules
\[
0\to \GW(k)\to \ZZ\oplus \mathrm{W}(k)\to \mathrm{W}(k)/2\to 0
\]
Localization being exact, we deduce an exact sequence of $S_{\psi}^{-1}\mathrm{GW}(k)$-modules. The above proposition shows that 
\[
S_{\psi}^{-1}\mathrm{GW}(k)\simeq S_{\psi}^{-1}\ZZ\oplus \bigoplus_{P\in\mathcal P}\QQ
\]
and we are left to prove that $S_{\psi}^{-1}\ZZ\simeq \QQ$. Let $p$ be an odd prime.
 Then, $p-2$ is odd and $(2n+4)(2n+2)=4p(p-1)$ and therefore $p$ is invertible.
 As we already know that $2$ is also invertible, the result follows.
\end{proof}

\begin{num}\label{num:symplectic_bo}
Let $S=\Spec(\ZZp)$.
 We are now in a position to classify all stable symplectic operations
 on rational Milnor-Witt motivic cohomology $\HMW \QQ_S$ of degree $(4n,2n)$, for $n \in \ZZ$.

We first consider the case $n=0$. Recall the decomposition in plus and minus part of 
Milnor-Witt motivic cohomology:
$$
\HMW \QQ_S=\HMW \QQ_{S+} \oplus \HMW \QQ_{S-}=\HM \QQ_{S} \oplus \HW_{\QQ}.
$$
We consider the additive symplectic operation
 $\tilde \chi_{2i}^M:\ZZ \times B\Sp \rightarrow K\big(\QQ_S(i),2i\big)$
 defined in Paragraph \ref{num:chi_tilda_HM}. According to Proposition \ref{prop:desuspension_add_op_HM},
 we get a stable symplectic operation
 (Definition \ref{df:stable_op})
 on rational motivic cohomology by considering the following sequence:
$$
\left(\tilde \chi^M_0,\frac 2 {4!} \tilde \chi^M_4,\hdots,\frac 2 {4n!} \tilde \chi^M_{4n},\hdots\right).
$$
Applying Corollary \ref{cor:pre_comput_stable-op}, this uniquely corresponds to a map 
 $\tilde{\bo}^+_0:\KSp_S\rightarrow \HM \QQ_S$ and we deduce a map:
$$
\tilde \bo_0:\KSp_S\xrightarrow{\tilde{\bo}^+_0} \HM \QQ_S \xrightarrow{i_+} \HMW \QQ_S
$$
where the last map is the canonical inclusion.
 For any $n \in \ZZ$, we put 
$$
\tilde \bo_{4n}:\KSp_S \simeq \KSp_S(4n)[8n] \xrightarrow{\tilde \bo_0(4n)[8n]} \HMW \QQ_S(4n)[8n].
$$
By definition, when $n>0$, it is induced as above by the following stable symplectic operation 
 on rational motivic cohomology of degree $(4n,8n)$:
\begin{equation}\label{eq:pre-borel_even}
\left(\frac 2 {4n!} \tilde \chi^M_{4n},\frac 2 {(4n+4)!} \tilde \chi^M_{4n+4},\hdots\right).
\end{equation}
For $n<0$, just add enough zeroes at the start.

Next we consider the case $n=2$.
 Let us consider the following product:
\begin{equation}\label{eq:quadratic_factorial}
\psi_{2+4n}!=\psi_2 \cdot \psi_6 \cdot \hdots \cdot \psi_{2+4n}.
\end{equation}
 Applying the previous corollary, this product of quadratic forms is invertible
 in $\HMW^{00}(S,\QQ)=\GW(S)_\QQ=\GW(\QQ)_\QQ$.
 Thus, using Theorem \ref{thm:desuspension_add_op_HMW},
 we can introduce the following stable symplectic operation on rational Milnor-Witt cohomology
 over $S$:
\begin{equation}\label{eq:pre-borel_two}
\left(\tilde \chi^\QQ_2,\frac 1 {\psi_6!} \tilde \chi^\QQ_6,\hdots,\frac 1 {\psi_{2+4n}!} \tilde \chi^\QQ_{2+4n},\hdots\right).
\end{equation}
Applying again Corollary \ref{cor:pre_comput_stable-op}, it uniquely corresponds to a morphism:
$$
\tilde \bo_2:\KSp_S \rightarrow \HMW \QQ_S(2)[4].
$$
For $n \in \ZZ$, we put 
$$
\tilde \bo_{2+4n}:\KSp_S \simeq \KSp_S(4n)[8n] \xrightarrow{\tilde \bo_2(4n)[8n]} \HMW \QQ_S(2+4n)[4+8n],
$$
which by definition is induced by the following stable symplectic operation, for $n>0$:
\begin{equation}\label{eq:pre-borel_odd}
\left(\frac 1 {\psi_{2+4n}!} \tilde \chi^\QQ_{2+4n},\frac 1 {\psi_{6+4n}!} \tilde \chi^\QQ_{6+4n},\hdots\right).
\end{equation}
Given now any $\ZZp$-scheme $S$, the operations $\tilde \bo_{2n}$ defined above can be defined over $S$
 by taking pullback along the unique morphism $S \rightarrow \Spec\left(\ZZp\right)$.
\end{num}
\begin{thm}\label{thm:stable_symplectic_op}
Let $S=\Spec\left(\ZZp\right)$ or $S=\Spec(k)$ with $k$ a perfect field of characteristic not $2$.
\begin{enumerate}
\item Let $n \in \ZZ$ be an even integer. Then one has canonical isomorphisms:
\begin{align*}
\Hom_{\mathrm{St}}\big(KSp_0,\wCH^{2n}_\QQ\big) 
 &\simeq \Hom_{\SH(S)}\left(\KSp_S,\HMW \QQ_S(2n)[4n]\right) \\
 &\simeq \Hom_{\SH(S)}\left(\KSp_S,\HM \QQ_S(2n)[4n]\right)
 \simeq \QQ
\end{align*}
where the first isomorphism is defined in Corollary \ref{cor:pre_comput_stable-op},
 and the second one is the projection on the plus part.
 Moreover, these $\QQ$-vector spaces are generated by the stable operation $\tilde \bo_{2n}$ defined above. \\
\item Let $n \in \ZZ$ be an odd integer. Then one has canonical isomorphisms:
\begin{align*}
\Hom_{\mathrm{St}}\left(KSp_0,\wCH^{2n}_\QQ\right) 
 &\simeq \Hom_{\SH(S)}\left(\KSp_S,\HMW \QQ_S(2n)[4n]\right) \\
 &\simeq \GW(S)_\QQ=\QQ \oplus \W(S)_\QQ
\end{align*}
where the first isomorphism is defined in Corollary \ref{cor:pre_comput_stable-op},
 and these $\GW(S)_\QQ$-modules are generated by the stable operation $\tilde \bo_{2n}$ defined above.
\end{enumerate}
\end{thm}
\begin{proof}
Each statement follows simply from Corollary \ref{cor:pre_comput_stable-op}
 and the computation of the projective system \eqref{eq:pre_comput_stable-op},
 whose transition maps are given by Theorem \ref{thm:desuspension_add_op_HMW}
 (note that by $(8,4)$-periodicity of $\KSp_S$,
 one can reduce to the case $n \geq 0$).
 By construction, we have for any symplectic bundle $\cU$ over $S$
equalities $\tilde \bo_0(\cU)=\tilde \chi_0^\QQ(\cU)=\rk(\cU)$
 and $\tilde \bo_2(\cU)=\tilde \chi_2^\QQ(\cU)=b_1(\cU)$ -- apply \eqref{eq:determinant_tilde_chi}.
 This implies that the operations $\tilde \bo_{2n}$ are non-zero.
\end{proof}

\begin{num}
Assume $S=\Spec(\ZZp)$ (or $\Spec(k)$).
Recall from \cite[6.2.3.9]{Riou} that for any $n \geq 0$, 
 one has:
$$
\Hom_{\SH(S)}(\KGL,\HM \QQ_S(n)[2n]) \simeq \QQ
$$
 is generated by the $n$-th component of the Chern character map
 $\ch_n:\KGL \rightarrow \HM \QQ_S$ (see footnote \ref{fn:RiouZ} page \pageref{fn:RiouZ}
 for the case $S=\Spec(\ZZp)$).
 Using the notations of Paragraph \ref{num:chi_tilda_HM},
 the map $\ch_n$ can be viewed as an $H$-stable operation as:
\begin{equation}\label{eq:Borel_operations}
\left(\frac 1 {n!}.\chi_n,\frac 1 {(n+4)!}.\chi_{n+4},\hdots\right).
\end{equation}
\end{num}
\begin{prop}\label{prop:compare_borel&chern}
Under the assumptions of the previous theorem, the following assertions hold.
\begin{enumerate}
\item For any even integer $n$, the following diagram is commutative:
$$
\xymatrix@C=24pt@R=14pt{
\KSp\ar^-{\tilde \bo_{2n}}[r]\ar_f[d] & \HMW \QQ_S(2n)[4n] \\
\KGL\ar^-{\ch_{2n}}[r] & \HM \QQ_S(2n)[4n]\ar_{i_+}[u]
}
$$
where $f$ is the forgetful map and $i_+$ the inclusion of the plus-part.
\item For any integer $n$, the following diagram is commutative:
$$
\xymatrix@C=24pt@R=14pt{
\KSp\ar^-{\tilde \bo_{2n}}[r]\ar_f[d] & \HMW \QQ_S(2n)[4n]\ar^{p_+}[d] \\
\KGL\ar^-{\ch_{2n}}[r] & \HM \QQ_S(2n)[4n]
}
$$
where $f$ is the forgetful map and $p_+$ the projection onto plus-part.
\end{enumerate}
\end{prop}
\begin{proof}
The first point follows directly from comparing formulas 
 \eqref{eq:pre-borel_even} and \eqref{eq:Borel_operations}
 using Lemma \ref{lm:compare_operations_BGL}.
Consider the second point. The case $n$ even is implied by the first point.
 The case $n$ odd reduces to $n>0$ (in fact $n=1$ is enough).
 Then one can compare formulas \eqref{eq:pre-borel_odd} and \eqref{eq:Borel_operations}
 using Lemma \ref{lm:compare_operations_BGL} and the fact: $\rk(\psi_{2n}!)=(2n)!/2$
 (use Formula \eqref{eq:denomin_Borel_char})
\end{proof}

\section{The Borel character}\label{sec:Borel}


\subsection{Definition and main theorem}

\begin{num}
We rephrase in the next statement the main theorem of the previous section,
 Theorem \ref{thm:stable_symplectic_op},
 in terms of higher Grothendieck-Witt groups.
 Let $S$ be a $\ZZp$-scheme $S$ and $n \in \ZZ$ be an integer.
 Using the isomorphism \eqref{eq:KSp&GW} and the notations of Paragraph \ref{num:symplectic_bo},
 we introduce the following maps. When $n$ is even,
\begin{align*}
\bo_{2n}:&\KQ_S \simeq \KSp_S(-2)[-4] \\
 &\xrightarrow{\tilde \bo_{2+2n}(-2)[-4]} \HMW \QQ_S(2+2n)[4+4n](-2)[-4]
 \simeq \HMW \QQ_S(2n)[4n]
\end{align*}
and when $n$ is odd:
\begin{align*}
\bo_{2n}:&\KQ_S \simeq \KSp_S(-2)[-4] \\
 &\xrightarrow{p_+ \circ \bo_{2+2n}(-2)[-4]} \HM \QQ_S(2+2n)[4+4n](-2)[-4]
 \simeq \HM \QQ_S(2n)[4n]
\end{align*}
Note that is follows from the construction of the stable $\Sp$-operation $\tilde \bo$ 
 (see \emph{loc. cit.}) that for any integer $n \in \ZZ$, one has:
\begin{equation}\label{eq:bo&suspension}
\begin{split}
\bo_{4n}:&\KQ_S \simeq \KQ_S(4n)[8n] \xrightarrow{\bo_0(4n)[8n]} \HMW \QQ_S(4n)[8n], \\
\bo_{2+4n}:&\KQ_S \simeq \KQ_S(4n)[8n] \xrightarrow{\bo_2(4n)[8n]} \HMW \QQ_S(2+4n)[4+4n],
\end{split}
\end{equation}
where the first isomorphism in each line is obtained by the periodicity of
 hermitian K-theory.\footnote{Which is given by the cup-product with the element $\kappa_S$ 
 recalled in Example \ref{ex:periodicity}.}
\end{num}
\begin{prop}\label{prop:stable_maps_KQ_HMW}
Consider the assumptions of the previous theorem.
 Then for any integer $n \in \ZZ$, one has:
$$
\Hom_{\SH(S)}\left(\KQ_S,\HMW \QQ_S(2n)[4n]\right)=
\begin{cases}
\QQ.(i_+ \bo_{2n}) & |n| \text{ odd}, \\
\GW(S)_\QQ.\bo_{2n} & |n| \text{ even}.
\end{cases}
$$
In the odd case, any map $\KQ_S \rightarrow \HMW \QQ_S(2n)[4n]$
 is zero on the minus part.
 The map $\bo_0:\KQ_S \rightarrow \HMW \QQ_S$
 is the unique map such that $bo_0(1_{GW})=1_{\HMW}$.
\end{prop}
\begin{proof}
Everything except the statement about $\bo_0$ follows from Theorem \ref{thm:stable_symplectic_op}.
 By definition of $\bo_0$, we get the following commutative diagram:
$$
\xymatrix@R=10pt@C=20pt{
\GW_0(S)\ar^{\bo_{0*}}[r]\ar_{(1)}[d] & \wCH^0(S) \\
\tilde \KSp(\HP^1_S)\ar@{^(->}[d] & \\
\KSp(\HP^1_S)\ar^{\tilde \bo_{2*}}[r] & \wCH^2(\HP^1_S)\ar_{(2)}[uu]
}
$$
where $(1)$ is the exterior product with the tautological symplectic bundle $\cH$ on $\HP^1$,
 and (2) is the projection on the factor $\wCH^0(S).b_1(\cH)$, using the symplectic projective
 bundle theorem for Chow-Witt groups (here $b_1$ denotes the Borel class for Chow-Witt groups).
 Thus the statement simply follows from the fact $\tilde \bo_{2*}=\tilde \chi_2=b_1$
 --- formula \eqref{eq:pre-borel_two} (resp. \eqref{eq:determinant_tilde_chi})
 for the first (resp. second) equality.
\end{proof}

\begin{df}\label{df:Borel_char}
Let $S$ be a $\ZZp$-scheme.
 We define the \emph{Borel character} over $S$ as the map:
\begin{align*}
\bo_t:\KQ^\QQ_{S} \xrightarrow{\ (\bo_{2n})_{n \in \ZZ}\ }
 & \bigoplus_{n \text{ even}} \HMW \QQ_S(2n)[4n] \oplus \bigoplus_{n \text{ odd}} \HM \QQ_S(2n)[4n] \\
 & \!\!\!\simeq \bigoplus_{n \in \ZZ} \SSp_{S,\QQ+}(2n)[4n] \oplus \bigoplus_{n \in \ZZ} \SSp_{S,\QQ-}(4n)[8n]
\end{align*}
\end{df}
Note that the map $(bo_{2n})_{n \in \ZZ}$ lands into a direct sum rather than a product
 because the corresponding product is in fact isomorphic to the above direct sum. This follows
 as $\SH(\ZZp)$ is compactly generated by objects of the form $\Sigma^\infty X_+(q)[p]$
 for $X$ smooth over $\ZZp$
 and from the fact that $\HMW^{4n-p,2n-q}(X,\QQ)$ vanishes for $n>>0$ or $n<<0$.
 The second isomorphism follows from the identifications \eqref{eq:MW&M_plusminus}.

\begin{num}\label{num:Borel_ppties}
Note that, according to the properties of $\tilde \bo_{2n}$, $\bo_t$ is compatible with
 pullbacks: it is a morphism of $\ZZp$-absolute spectra.
Using Proposition \ref{prop:compare_borel&chern}, we immediately
 get the commutativity of the following square:
$$
\xymatrix@C=24pt@R=14pt{
\KQ^\QQ_{S}\ar^-{\bo_t}[r]\ar_f[d] & \bigoplus_{n \text{ even}} \HMW \QQ_S(2n)[4n] \oplus \bigoplus_{n \text{ odd}} \HM \QQ_S(2n)[4n]\ar[d] \\
\KGL^\QQ_{S}\ar^-{\ch_t}[r] & \bigoplus_{m \in \ZZ} \HM \QQ_S(m)[2m]
}
$$
where $f$ is the natural forgetful map, and the right-hand vertical map
 is obtained by the projection on the plus part for the even integer $m$,
 $0$ for $m$ odd. In particular, for odd $n$, the map $\bo_{2n}=\ch_{2n} \circ f$
 (apply point (1) of Proposition \ref{prop:compare_borel&chern}).
\end{num}

In the remaining of this section, we will prove that
 for any $\ZZp$-scheme $S$,
 the Borel character $bo_t$ is an isomorphism of ring spectra:
 see Theorem \ref{thm:Borel_iso}.

\subsection{Principle of proof and periodic spectra}\label{sec:periodic}

\begin{num}\textit{Principle of proof of Th. \ref{thm:Borel_iso}}.
To prove that the Borel character $\bo_t$ is an isomorphism over any scheme $S$,
 by compatibility with pullbacks (see Paragraph \ref{num:Borel_ppties}),
 it is sufficient to consider the case $S=\Spec(\ZZp)$.
 Moreover we can always use Morel's decomposition of $\SH(S)_\QQ$ into its plus and minus parts.

That being said, we will prove that $\bo_t$ is an isomorphism by
 constructing an explicit inverse $\bo'_t$ using the theory of periodic ring spectra (see definition below).
 The advantage of this construction is that $\bo'_t$ will clearly be a morphism of ring spectra.
 We will construct $\bo'_t$ by separating the plus part (section) and the minus part (section ).
\end{num}

The following result is classical (see e.g. \cite[Prop. 2.6]{GepSna}) in (motivic) homotopy theory:
\begin{prop}
Let $\E$ be a motivic ring spectrum over $S$.
 Consider a pair of integers $(n,i) \in \ZZ^2$.
 Then the following conditions are equivalent:
\begin{enumerate}
\item[(i)] There exists an element $\rho \in \E_{n,i}(S)$, invertible for the cup-product
 on $\E^{**}$.
\item[(ii)] There exists an isomorphism: $\tilde \rho:\E(i)[n] \rightarrow \E$.
\end{enumerate}
\end{prop}

\begin{df}
A pair $(\E,\rho)$ satisfying the equivalent conditions of the above proposition will be called an \emph{$(n,i)$-periodic ring spectrum} over $S$.

 We obviously get a $\Sch_B$-fibered category of periodic ring spectra.
 A \emph{$B$-absolute $(n,i)$-periodic ring spectrum $(\E_,\rho)$} is a section of this $\Sch_B$-fibered category.\footnote{\emph{i.e.}
 a collection of periodic $(n,i$)-spectra $(\E_S,\rho_S)$ for schemes $S$ in $\Sch_B$
 such that for any morphism $f:T \rightarrow S$, $f^*(\rho_S)$ corresponds to $\rho_T$
 via the given isomorphism $f^*(\E_S) \simeq \E_T$.}
\end{df}
Given such a periodic absolute ring spectrum, we get a universal morphism of absolute ring spectra:
\begin{equation}\label{eq:periodization_map}
\sigma_\rho:\bigoplus_{r \in \ZZ} \SSp_S(ri)[rn] \xrightarrow{\sum_r \rho^r} \E_S
\end{equation}
with source the $(n,i)$-periodization of the sphere spectrum.

\begin{ex}\label{ex:periodicity}
\begin{enumerate}
\item The $K$-theory spectrum $\KGL$ together with the Bott element $\beta$ is $(2,1)$-periodic,
 as an absolute ring spectrum over $\ZZ$. Note for normalization purposes that we consider
 $\beta_\ZZ$ as the element of $\KGL_{2,1}(\ZZ)$ uniquely defined by the following property:
$$
\xymatrix@=10pt
{
\KGL_{2,1}(\ZZ)\ar@{=}[r] & \lbrack\un_\ZZ(1)[2],\KGL_\ZZ\rbrack^{st}\ar@{}|/6pt/\subset[r] & \lbrack\PP^1_\ZZ,\KGL_\ZZ\rbrack^{st}\ar^-\sim[r] &K_0(\PP^1_\ZZ) \\
 \beta_\ZZ\ar@{|->}[rrr] &&& [\mathcal O(-1)]-[\mathcal O]
}
$$
where $\mathcal O(-1)$ is the tautological line bundle on $\PP^1_S$.

Note that it follows from the relation $\ch_n(\beta^i)=\delta_{ni}$ that the rational and plus part of the periodization map:
$$
\sigma_\beta^{\QQ+}:\bigoplus_{n \in \ZZ} \HM\QQ_S(n)[2n] \rightarrow \KGL_{S,\QQ}
$$
is an isomorphism of ring spectra with inverse the Chern character map:
$$
\ch_t:\KGL_S \xrightarrow{(\ch_n)} \bigoplus_{n \in \ZZ} \HM\QQ_S(n)[2n].
$$
\item By construction (see \cite{PW1}), the hermitian K-theory spectrum $\KQ$ is $(8,4)$-periodic, as an absolute ring spectrum over $\ZZp$.
 We will consider the periodicity element $\kappa \in \KQ_{8,4}(\ZZp)$ characterized by the property:\footnote{While $\beta$ refers
 for Bott periodicity, the choice of letter $\kappa$ refers to Karoubi theorem which implies the periodicity of the spectrum $\KQ$
 (see \cite[Th. 6.2]{Schlicht2}).}
$$
\xymatrix@=10pt
{
\KQ_{8,4}(\ZZp)\ar@{=}[r] & \lbrack\un(4)[8],\KQ\rbrack^{st}\ar@{}|/9pt/\subset[r] & \GW^{0,0}(P)\ar^-\sim[r] &KO_0(P) \\
 \kappa_\ZZ\ar@{|->}[rrr] &&& [\cU \otimes \cU]-[\mathcal O].
}
$$
where $P=\HP^1_{\ZZp} \times \HP^1_{\ZZp}$, $\cU$ is the tautological symplectic bundle on $\HP^1_{\ZZp}$ (see Paragraph \ref{num:sp_pb_thm}),
 and $\cU \otimes \cU$ is the external product, seen as a quadratic bundle.
 Explicitly, $\kappa=(\cU_1-\cH)\otimes (\cU_2-\cH)$.

It follows from our conventions that for any $\ZZp$-scheme $S$,
  the forgetful map $f:\KQ_S \rightarrow \KGL_S$ sends $\kappa_S$ to $\beta_S^4$.
\item The $\ZZ[1/2]$-absolute ring spectrum $\KW$, representing Balmer's higher Witt groups, together with the Hopf map $\eta$
 is $(1,1)$-periodic (see Section \ref{sec:minus}).
\end{enumerate}
\end{ex}

\subsection{The plus part}

\begin{num}
Given an arbitrary $\ZZp$-scheme $S$,
 it follows from \cite[Th. 6.1]{Schlicht2} and \cite[Th. 3.4]{RonOst} that one has a canonical distinguished triangle:
$$
\KQ_S(1)[1] \xrightarrow{\eta} \KQ_S \xrightarrow f \KGL_S \xrightarrow{h \circ \gamma_{\beta'}} \KQ_S(1)[2]
$$
where $\eta$ is the Hopf map, $f$ the forgetful map,
 $h:\KGL_S \rightarrow \KQ_S$ is the "hyperbolization" map
 and $\gamma_{\beta'}$ is the multiplication by the inverse of the Bott element $\beta$
 (Example \ref{ex:periodicity}).
 As $\eta_+=0$ and $\KGL_{S-}=0$, we immediately deduce the following result.
\end{num}
\begin{lm}
The following exact sequence is split exact in $\SH(S)_+$:
$$
0 \rightarrow \KQ_{S+}
 \xrightarrow f \KGL_S[2^{-1}] \xrightarrow{h \circ \gamma_{\beta'}} \KQ_{S+}(1)[2] \rightarrow 0.
$$
In other words, $\KGL_S[2^{-1}] \simeq \KQ_{S+} \oplus \KQ_{S+}(1)[2]$.
\end{lm}
There is moreover a canonical splitting of the above exact sequence. Indeed, according to \cite[3.6]{RonOst},
 one gets the relation in $\End_{\SH(S)_+}(\KQ_{S+})$:
\begin{equation}\label{eq:forgetful&splitting}
h \circ f=1-\epsilon_+=2
\end{equation}
as, by design, $\epsilon_+=-1$.

\begin{num}
By construction of the hermitian K-theory spectrum and the forgetful map $f$ (see \cite[Prop. 3.3]{RonOst}),
 we get the following commutative diagram:
$$
\xymatrix@R=12pt@C=24pt{
\KQ_{4,2}(S)\ar_f[d] & \KSp_0(S)\ar^f[d]\ar^\phi_-\sim[l] \\
\KGL_{4,2}(S) & \KGL_0(S),\ar_-{\gamma_{\beta^2}}[l]
}
$$
where $\gamma_{\beta^2}$ is multiplication by $\beta^2$, and $f$ on the right-hand side
 is the map forgetting the hermitian structure. We these notation,
 we can define the following element in the positive part of hermitian K-theory:
$$
\rho_S=\frac 12.\phi([\cH]) \in \KQ^+_{4,2}(S).
$$
This elements is stable under pullback so we can erase the base $S$ to simplify notation.
It follows from the above commutative square that $f(\rho)=\beta^2$.

The same construction can be done replacing in degree $(-4,-2)$ by replacing $\beta$ with $\beta^{-1}$.
 Thus we get an element $\rho' \in \KQ^+_{-4,-2}(S)$ such that $f(\rho')=\beta^{-2}$.

Let us remark that the forgetful map $f:\KQ_S \rightarrow \KGL_S$ is a morphism of ring spectra.
 This implies that $f(\rho.\rho')=1$. As $f$ is a monomorphism on the plus-part,
 due to \eqref{eq:forgetful&splitting}, we deduce that $\rho$ is invertible.
 Thus $(\KQ_+,\rho)$ is a periodic absolute $(4,2)$-spectrum over $\ZZp$. In particular we
 get a canonical map by taking the plus part of \eqref{eq:periodization_map} :
$$
\bigoplus_{n \in \ZZ} \SSp_{S,+}(2n)[4n] \xrightarrow{\sum_n \rho^n} \KQ_{S+}.
$$

Note moreover that we get the following relations for any integer $n \in \ZZ$:
\begin{equation}\label{eq:rho&beta}
f(\rho^n)=\beta^{2n}, \quad 
h(\beta^{n})=\begin{cases}
2.\rho^i & n=2i, \\
0 & n \text{ odd}.
\end{cases}
\end{equation}
The first relation follows as $f$ is a morphism of ring spectra,
 and the second one from the first and the fact that $h \circ \gamma_{\beta'} \circ f=0$
 (the above lemma). We immediately deduce from these relations the following lemma.
\end{num}
\begin{lm}
The following diagram is commutative in $\SH(S)_+$:
$$
\xymatrix{
\KQ_{S+}\ar^f[r] & \KGL_S[1/2]\ar^{h \circ \gamma_{\beta'}}[r] & \KQ_{S+}\dtw 1 \\
\bigoplus_{n \in \ZZ} \SSp_{S+}\dtw{2n}\ar^-{\sigma_\rho}[u]\ar@{^(->}[r] & \bigoplus_{m \in \ZZ} \SSp_{S+}\dtw m\ar@{->>}[r]\ar|/-2pt/{\sigma_\beta}[u]
 & \bigoplus_{n \in \ZZ} \SSp_{S+}\dtw{2n+1}.\ar_-{\sigma_\rho\dtw 1}[u]
}
$$
\end{lm}

As an application, we get the following result which conclude the "plus-part" of Theorem \ref{thm:Borel_iso}.
\begin{prop}\label{prop:Borel_plus}
Consider the above notation.
 The morphism of rational ring spectra
$$
\sigma_\rho:\bigoplus_{n \in \ZZ} \SSp_{S,\QQ+}(2n)[4n] \rightarrow \KQ_{S,\QQ+}
$$
is an isomorphism, and the following relation hold:
$$
\bo_{t+} \circ \sigma_\rho=1.
$$
\end{prop}
The first assertion follows from the preceding lemma and point (1) of Example~\ref{ex:periodicity}.
 The second assertion follows easily from relation \eqref{eq:rho&beta},
 the commutativity of the square in Paragraph~\ref{num:Borel_ppties}
 and again Example~\ref{ex:periodicity}(1).

\subsection{The minus part}\label{sec:minus}

\begin{num}
Let $S$ be a $\ZZp$-scheme.
According to our conventions, there exists a natural map
\[
\KQ_S\to \KW_S
\] 
which induces an isomorphism:
$$
\KQ_{S,-}\rightarrow \KW_S[1/2].
$$
In particular, $\KW_S[1/2]$ is $(1,1)$-periodic with respect to $\eta$
 and $(8,4)$-periodic with respect to $\kappa_- \in \KW_{8,4}$ which is the image
 of $\kappa$ under the canonical projection :$\KQ_S \rightarrow \KQ_{S-}$.
 The next result generalizes the fundamental result of \cite{ALP}
 to an arbitrary base $\ZZp$-scheme.
\end{num}
\begin{prop}\label{prop:Borel_minus1}
The canonical map associated to $\kappa_- \in \KW_{8,4}(S)$ as in \eqref{eq:periodization_map}:
$$
\sigma_{\kappa_-}:\bigoplus_{n \in \ZZ} \SSp_{Q-}(4n)[8n] \rightarrow \KW_S[1/2]
$$
is an isomorphism of ring spectra.
\end{prop}
\begin{proof}
The map $\sigma_{\kappa_-}$ is compatible with base change.
 Recall from \cite[4.3.17]{CD3} that the family of functors $x^*$ for $x \in S$ a point of $X$
 is conservative. Thus we are reduced to the case where $S$ is the spectrum of a field
 (of characteristic different from $2$).
 Then the result can be reduced to \cite[Corollary~3.5]{ALP}.
 In fact, our isomorphism is the inverse of that of \emph{loc. cit.}.
 Indeed, the stable operations $\rho_m^{st}$ that compose the latter
 are defined by the relation: $\rho_m^{st}(\kappa_-^n)=\delta_{nm}.\kappa_-^n$;
 see \cite[Definition~2.5]{ALP}, given that the element $\beta$ in \emph{loc. cit.}
 is our element $\kappa_-$.
\end{proof}

Here is the last result needed to conclude the proof of Theorem~\ref{thm:Borel_iso}.
\begin{prop}\label{prop:Borel_minus2}
For any integers $n,i \in \ZZ$, the following relation holds
 in $\HMW^{8(n-i),4(n-i)}(S)$:
$$
\bo_{4n-}(\kappa^i_-)=\delta_n^i.
$$
\end{prop}
\begin{proof}
Note first that by compatibility with pullbacks, it is sufficient
 to treat the case where $S=\Spec(\ZZp)$.
 According to the first of the relations \eqref{eq:bo&suspension},
 it is sufficient to treat the case $n=0$.
 The case $i=0$ follows from the last assertion of Proposition \ref{prop:stable_maps_KQ_HMW}.
 The vanishing for $n=0$, $i\neq 0$ follows as MW-motivic cohomology of $\ZZp$ vanishes
 in degree $(8r,4r)$ for $r=(n-i)\neq 0$, as recalled in Paragraph \ref{num:HMW_vanishing}.
\end{proof}

\subsection{Conclusion}

Putting together Propositions \ref{prop:Borel_plus}, \ref{prop:Borel_minus1} and \ref{prop:Borel_minus2},
 we get the main theorem of this section:
\begin{thm}\label{thm:Borel_iso}
Let $S$ be an arbitrary $\ZZp$-scheme.
 Then the Borel character $\bo_t$ (Definition \ref{df:Borel_char}) is an isomorphism of ring spectra with reciprocal
 isomorphism:
$$
\sigma_\rho+\sigma_{\kappa_-}:\bigoplus_{n \in \ZZ} \SSp_{S,\QQ+}(2n)[4n] \oplus \bigoplus_{n \in \ZZ} \SSp_{S,\QQ-}(4n)[8n]
 \rightarrow \KQ_{S,\QQ}.
$$
\end{thm}
Note by the way that, taking into account the relation $\rho^2=\kappa_+$, and the identifications \eqref{eq:MW&M_plusminus},
 one can rewrite the preceding isomorphism as:
$$
\sigma_\kappa+\sigma_{\kappa_+}.\rho:\bigoplus_{n \in \ZZ} \HMW \QQ_{S}(4n)[8n] \oplus \bigoplus_{n \in \ZZ} \HM \QQ_S(4n+2)[8n+4]
 \rightarrow \KQ_{S,\QQ}.
$$

\begin{rem}
As indicated to us by F. Jin,
 it can be noted that the reciprocal isomorphism of the Borel character exists integrally.
 This is not the case of the Borel character, as for the Chern character.
\end{rem}

\section{Applications}\label{sec:applications}


\subsection{A Suslin-type homomorphism}\label{sec:suslinhomomorphism}

In this section $k$ is a perfect field of characteristic different from $2$. We begin with the construction of a slightly different model of the spectrum $\KQ$. With this in mind, we can use \cite[Theorem 1.3]{ST} (in the spirit of \cite[\S 2.2]{Asok14b}) to observe that the Orthogonal Grassmannian $OGr$ constructed in \emph{loc. cit.} admits explicit deloopings $\Omega_{\PP^1}^{-n}(\ZZ\times OGr)$ for $n\in\NN$ satisfying $[\Sigma_{S^1}^i(X_+),\Omega_{\PP^1}^{-n}(\ZZ\times OGr)]_{\AA^1}\simeq \GW^n_i(X)$ for any smooth scheme $X$. In particular, $\Omega_{\PP^1}^{-2}(\ZZ\times OGr)\simeq \ZZ\times HGr$. This allows to define a spectrum whose term in degree $n$ is $\Omega_{\PP^1}^{-n}(\ZZ\times OGr)$ and whose bonding maps are the adjoints of the equivalences
\[
\Omega_{\PP^1}^{-n}(\ZZ\times OGr)\simeq \Omega_{\PP^1}\Omega_{\PP^1}^{-n-1}(\ZZ\times OGr).
\]
We still denote this spectrum by $\KQ$, as it is canonically isomorphic to the one we considered before. The unit map $\varepsilon:\un\to \KQ$ of this spectrum was explicitly constructed in \cite{Asok14b}, yielding explicit morphisms
\[
\varepsilon_n:(\PP^1)^{\wedge n}\to \Omega_{\PP^1}^{-n}(\ZZ\times OGr)
\]
for each $n\in \NN$. For an integer $n\in\NN$, we define $\mu(n)\in \NN$ to be the smallest integer of the form $4m+2$ greater than or equal to $n$. Note in particular that $\mu(4n+2)=4n+2$ for any $n\in\NN$. We can define an operation
\[
\tilde\chi_{n}:\Omega_{\PP^1}^{-n}(\ZZ\times OGr)\to K\big(\tilde\ZZ(n),2n\big)
\]
for any $n\in\NN$
using the commutative diagram
\[
\xymatrix{\Omega_{\PP^1}^{-n}(\ZZ\times OGr)\ar[r]\ar@{-->}[d]_-{\tilde\chi_{n}} & \Omega_{\PP^1}^{\mu(n)-n}\Omega_{\PP^1}^{-\mu(n)}(\ZZ\times OGr)\ar[d]^{\Omega_{\PP^1}^{\mu(n)-n}\tilde\chi_{\mu(n)}} \\
K\big(\tilde\ZZ(n),2n\big)\ar[r] &  \Omega_{\PP^1}^{\mu(n)-n}K\big(\tilde\ZZ(\mu(n)),2\mu(n)\big)}
\]
in which the vertical maps are isomorphisms. Altogether, we get a composite 
\[
(\PP^1)^{\wedge n}\stackrel{\varepsilon_n}\to \Omega_{\PP^1}^{-n}(\ZZ\times OGr)\stackrel{\tilde\chi_{n}}\to K\big(\tilde\ZZ(n),2n\big)
\]
which induces homomorphisms of sheaves 
\[
\piaone_i((\PP^1)^{\wedge n})\to \piaone_i(\Omega_{\PP^1}^{-n}(\ZZ\times OGr))\to \piaone_i(K\big(\tilde\ZZ(n),2n\big))
\]
for each $i\in\NN$. As $[\Sigma_{S^1}^i(X_+),\Omega_{\PP^1}^{-n}(\ZZ\times OGr)]_{\AA^1}\simeq \GW_i^n(X)$ for any smooth scheme $X$, we see that $\piaone_i(\Omega_{\PP^1}^{-n}(\ZZ\times OGr))$ is the sheaf associated to the presheaf $\GW_i^n$, while $\piaone_i(K\big(\tilde\ZZ(n),2n\big))$ is the sheaf associated to the presheaf $X\mapsto \mathrm{H}^{2n-i}_{\mathrm{MW}}(X,\ZZ)$. Setting $i=n$ above and considering sections on a (finitely generated) field extension $L/k$, we then obtain a string of homomorphisms
\[
\piaone_i((\PP^1)^{\wedge n})(L)\xrightarrow{\varepsilon_{n,n}} \GW_n^n(L)\xrightarrow{\tilde\chi_{n,n}} \mathrm{H}_{\mathrm{MW}}^{n,n}(L,\ZZ)
\]
Suppose that $n\geq 2$. In light of \cite[Corollary 6.43]{MorLNM} and \cite[Theorem 4.2.3]{Deglise16}, we finally obtain homomorphisms 
\[
\KMW_n(L)\xrightarrow{\varepsilon_{n,n}} \GW_n^n(L)\xrightarrow{\tilde\chi_{n,n}}\KMW_n(L).
\]
We note that the first map coincides (up to a unit) with the map $\KMW_{4n+2}(L)\to \GW_{4n+2}^2(L)$ induced by the identity $\KMW_1(L)=\GW^1_1(L)$ and the ring structure on both sides (\cite[Theorem 4.2.2]{Asok14b}). 

\begin{thm}\label{thm:strongsuslin}
For any $n\geq 2$ and any finitely generated field extension of $k$, the composite
\[
\KMW_n(L)\stackrel{\varepsilon_{n,n}}\to \GW_n^n(L)\stackrel{\tilde\chi_{n,n}}\to\KMW_n(L)
\]
is multiplication by $\psi_{\mu(n)}!\in \GW(k)$
 according to formula \eqref{eq:quadratic_factorial},
 where $\mu(n)$ is the smallest integer congruent to 2 modulo 4 and greater than $n$.
\end{thm}

\begin{proof}
We have adjunctions
\[
\mathcal H_{\AA^1}(k) \leftrightarrows \mathrm{D}_{\AA^1}^{\mathrm{eff}}(k)  \leftrightarrows \widetilde{\mathrm{DM}}^{\mathrm{eff}}(k),
\]
the first one being the classical Dold-Kan correspondence and the second one being the adjunction of \cite[\S 3.2.4]{Deglise16}. We can thus consider the resulting adjunction 
\[
\mathcal H_{\AA^1}(k)\leftrightarrows \widetilde{\mathrm{DM}}^{\mathrm{eff}}(k).
\]
The image of the object $(\PP^1)^{\wedge n}$ of $\mathcal H_{\AA^1}(k)$ is precisely $\tilde\ZZ(n)[2n]$ and the unit of the adjunction induces a map
\[
\eta_n:(\PP^1)^{\wedge n}\to K\big(\tilde\ZZ(n),2n\big)
\]
which is in fact the degree $n$ morphism of the unit map of the spectrum $\HMW$. It is an explicit generator of $[(\PP^1)^{\wedge n},K\big(\tilde\ZZ(n),2n\big)]\simeq \GW(k)$ and we claim that it induces an isomorphism upon applying $\piaone_n$. In view of the Hurewicz homomorphism in $\AA^1$-homotopy theory and \cite[Theorem 6.37]{MorLNM}, we are reduced to show that the unit map 
\[
C^{\AA^1}_*((\PP^1)^{\wedge n})\to \tilde\ZZ(n)[2n]
\]
in the adjunction $\mathrm{D}_{\AA^1}^{\mathrm{eff}}(k)  \leftrightarrows \widetilde{\mathrm{DM}}^{\mathrm{eff}}(k)$ induces an isomorphism upon applying $\mathrm{H}_{n}$. This is tantamount to showing that the unit map
\[
C^{\AA^1}_*((\mathbb{G}_m)^{\wedge n})\to  \tilde\ZZ(n)[n]
\]
induces an isomorphism on $\mathrm{H}_0$, which is the case by \cite[Theorem 4.2.3]{Deglise16}.

Now, we have a commutative diagram
\[
\xymatrix{(\PP^1)^{\wedge 2}\ar[r]\ar[d]_{\varepsilon_2} & \Omega_{\PP^1}^{n-2} (\PP^1)^{\wedge n} \ar[d]^-{\Omega_{\PP^1}^{n-2}\varepsilon_{n}}\ar[r] &\Omega_{\PP^1}^{\wedge \mu(n)-2} (\PP^1)^{\mu(n)} \ar[d]^-{\Omega_{\PP^1}^{\mu(n)-2}\varepsilon_{\mu(n)}}\\
\ZZ\times HGr\ar[r]\ar[d]_-{\psi_{\mu(n)}!\tilde\chi_2} &  \Omega_{\PP^1}^{n-2}\Omega_{\PP^1}^{-n}(\ZZ\times OGr)\ar[d]^-{\Omega_{\PP^1}^{n-2}\tilde\chi_{n}} \ar[r] &  \Omega_{\PP^1}^{\mu(n)-2}\Omega_{\PP^1}^{-\mu(n)}(\ZZ\times Ogr)\ar[d]^-{\Omega_{\PP^1}^{\mu(n)-2}\tilde\chi_{\mu(n)}}\\
K\big(\tilde\ZZ(2),4\big)\ar[r] & \Omega_{\PP^1}^{n-2}K\big(\tilde\ZZ(n),2n\big)\ar[r] &  \Omega_{\PP^1}^{\mu(n)-2}K\big(\tilde\ZZ(\mu(n)),2\mu(n)\big)} 
\]
by (a repeated use of) Theorem \ref{thm:desuspension_add_op_HMW} and $\tilde\chi_2\circ\varepsilon_2$ corresponds to $b_1(\cU)$, which is in fact $\eta_2$. Thus $\tilde\chi_{n}\circ\varepsilon_{n}$ corresponds to $\psi_{\mu(n)}!\eta_2$ and consequently we have 
\[
\tilde\chi_{n}\circ\varepsilon_{n}=\psi_{\mu(n)}!\eta_n
\]
The right-hand side is in fact the composite 
\[
(\PP^1)^{\wedge n}\xrightarrow{\psi_{\mu(n)}!} (\PP^1)^{\wedge n}\xrightarrow{\eta_n} K\big(\tilde\ZZ(\mu(n)),2\mu(n)\big)
\]
where the first map (still denoted by $\psi_{\mu(n)}!$) is obtained via the isomorphism $[(\PP^1)^{\wedge n},(\PP^1)^{\wedge n}]_{\AA^1}\simeq \GW(k)$ of \cite[Corollary 6.43]{MorLNM}. 
\end{proof}

\begin{rem}\label{rem:suslin}
Let $\varepsilon^{\GL}:\un\to \KGL$ be the unit map of the spectrum representing $K$-theory. Using the operations $\chi$ from $K$-theory to motivic cohomology, one can repeat the above construction to get homomorphisms
\[
\KMW_n(L)\stackrel{\psi_{n,n}}\to \K_n(L)\stackrel{\tilde\chi_{n,n}}\to\KM_n(L)
\]
The first map factors through Milnor $K$-theory and we get a commutative diagram
\[
\xymatrix@C=4em{\KMW_{n}(L)\ar[r]^-{\varepsilon_{n,n}}\ar[d] &  \GW_{n}^n(L)\ar[r]^-{\tilde\chi_{n,n}}\ar[d]  &   \KMW_{n}(L)\ar[d]  \\
\KM_{n}(L)\ar[r]_-{\varepsilon^{\GL}_{n,n}} &  \K_{n}(L)\ar[r]_-{\chi_{n,n}} &   \KM_{n}(L) } 
\]
We note that the composite is then equal to $\mu(n)!$, showing that the homomorphism $\K_{n}(L)\to  \KM_{n}(L)$ is not optimal, in comparison with the map $\K_{n}(L)\to  \KM_{n}(L)$ defined by Suslin in \cite{Suslin84}. If $n=\mu(n)$, we guess that the above homomorphism coincides with the one defined by Suslin, up to a sign. 
\end{rem}

\begin{rem}
The homomorphism $\psi_{n,n}:\KMW_n(L)\to \GW_{n}^n(L)$ is actually an isomorphism for $n\leq 3$ by \cite[\S 4]{Asok14b}. An unpublished result of O. R\"ondigs states that this homomorphism is an isomorphism for $n=4$ as well. 
\end{rem}

\begin{cor}
Let $L$ be a field and let $n\geq 2$ be such that $\KMW_n(L)$ is $\psi_{\mu(n)}!$-torsion free. Then, $\KMW_n(L)$ injects into $\GW_n^n(L)$.
\end{cor}

\begin{rem}
As an example, we may consider $\RR$, or actually any real closed field. Indeed, we know from \cite[Proposition 2.2]{Kolderup19} that $\KMW_n(\RR)\simeq \ZZ\oplus D$, where $D$ is a uniquely divisible group. It follows that $\KMW_n(\RR)$ always injects into $\GW_n^n(\RR)$, extending \cite[Example 4.3.3]{Asok14b}. Another example is given by $F=\QQ$ for $n\geq 3$ by \cite[Proposition 2.4]{Kolderup19}, or by any algebraically closed field with $n\geq 2$.
\end{rem}

\section{Appendix: The threefold product} \label{sec:appendix}

The purpose of this appendix is to explicitly decompose the threefold product
\[
(U,\varphi)\otimes (U,\varphi)\otimes (U,\varphi)
\]
where $(U,\varphi)$ is a symplectic bundle of rank $2$. The first lemma is obvious and we omit its proof. 

\begin{lm}
Let $U$ be a rank $2$ bundle and let $i:\wedge^2U\to U\otimes U$ be the homomorphism given on sections by $s_1\wedge s_2\mapsto s_1\otimes s_2-s_2\otimes s_1$. Then, we have an exact sequence
\[
0\to \wedge^2U\stackrel{i}\to U^{\otimes 2}\to \mathrm{Sym}^2U\to 0.
\]
\end{lm}

Tensoring the above by $U$ (say on the right), we obtain an exact sequence
\[
0\to (\wedge^2U)\otimes U\xrightarrow{i\otimes 1} U^{\otimes 3}\to (\mathrm{Sym}^2U)\otimes U\to 0.
\]
Now, we may define a homomorphism $j:(\wedge^2U)\otimes U\to (\mathrm{Sym}^2U)\otimes U$ on sections by 
\[
(s_1\wedge s_2)\otimes s_3\mapsto (s_1 s_3)\otimes s_2 - (s_2 s_3)\otimes s_1.
\]

\begin{lm}
We have an exact sequence
\[
0\to (\wedge^2U)\otimes U\stackrel{j}\to (\mathrm{Sym}^2U)\otimes U\to \mathrm{Sym}^3U\to 0.
\]
\end{lm}

\begin{proof}
We may work locally, and thus suppose that $U=R^2$ (where $R$ is a local ring) with basis $e,f$. We then have a basis $e\wedge f$ of $(\wedge^2U)$ and a basis $(e\wedge f)\otimes e, (e\wedge f)\otimes f$ of $(\wedge^2U)\otimes U$. Their respective images are $(e e)\otimes f-(e f)\otimes e$ and $(e f)\otimes f-(f f)\otimes e$ which are linearly independent. Now, the composite of the morphisms are clearly trivial, and therefore the sequence is exact by an easy dimension count.
\end{proof}

\begin{lm}\label{lem:firstfactor}
Suppose that $k$ is of characteristic different from $2$. Then, the restriction of $\varphi^{\otimes 3}$ to $(\wedge^2U)\otimes U$ along  
\[
(i\otimes 1):(\wedge^2U)\otimes U\to U^{\otimes 3}
\] 
is isometric to $\langle 2\rangle \varphi:U\to U^\vee$. In particular, it is non degenerate.
\end{lm}

\begin{proof}
Again, we suppose that we work over a local ring $R$, and thus that $U$ has a basis $e,f$. The form $\varphi$ is then characterized by $\psi:\wedge^2U\to R$ given by $\psi(e\wedge f)=\varphi(e,f)$. An easy computation shows that 
\[
(i\otimes 1)^\vee(\varphi^{\otimes 3})(i\otimes 1)
\]
is characterized by $((e\wedge f)\otimes e)\wedge ((e\wedge f)\otimes f)\mapsto 2\varphi(e,f)^3$. 
We can now define an isomorphism $U\to (\wedge^2U)\otimes U$ by $s\mapsto \psi^{-1}(1)\otimes s$ and check that the restriction is precisely $\langle 2\rangle \varphi$.
\end{proof}

\begin{rem}
In characteristic $2$, the above lemma shows that $(\wedge^2U)\otimes U$ is a sublagrangian of $(U^{\otimes 3},\varphi^{\otimes 3})$.
\end{rem}

Consequently, we see that if $k$ is of characteristic different from $2$, then we get a decomposition $(U^{\otimes 3},\varphi^{\otimes 3})\simeq (U,\langle 2\rangle\varphi)\perp ((\mathrm{Sym}^2U)\otimes U,\varphi^\prime)$ for some form $\varphi^\prime$ that we now determine. We have an obvious homomorphism $\mathrm{Sym}^2U\to U^{\otimes 2}$ (always under the hypothesis that $k$ is of characteristic different from $2$) given by $s_1 s_2\mapsto s_1\otimes s_2+s_2\otimes s_1$. This induces a section of $ U^{\otimes 3}\to (\mathrm{Sym}^2U)\otimes U$ and we may consider the form induced by this section. However, it is easier to consider the composite of $(\wedge^2U)\otimes U\stackrel{j}\to (\mathrm{Sym}^2U)\otimes U$ with this section and the form induced on $(\wedge^2U)\otimes U$.

\begin{lm}
Suppose that $k$ is of characteristic different from $2,3$. The symplectic form induced on $\wedge^2U\otimes U$ by $(U^{\otimes 3},\varphi^{\otimes 3})$ under the homomorphism
\[
m:\wedge^2U\otimes U\to U^{\otimes 3}
\]
defined by $(s_1\wedge s_2)\otimes s_3\mapsto s_1\otimes s_3\otimes s_2+s_3\otimes s_1\otimes s_2-s_2\otimes s_3\otimes s_1-s_3\otimes s_2\otimes s_1$ is isometric to $\langle 6\rangle(U,\varphi)$.
\end{lm}

\begin{proof}
The same proof as in Lemma \ref{lem:firstfactor} shows that the form on $\wedge^2U\otimes U$ is locally characterized by $((e\wedge f)\otimes e)\wedge ((e\wedge f)\otimes f)\mapsto 6\varphi(e,f)^3$.
\end{proof}

As a consequence, we have 
\begin{equation}\label{eqn:decomposition}
(U^{\otimes 3},\varphi^{\otimes 3})\simeq (U,\langle 2\rangle\varphi)\perp (U,\langle 6\rangle\varphi) \perp (\mathrm{Sym}^3U,\psi)
\end{equation}
where $\psi$ is induced by $(U^{\otimes 3},\varphi^{\otimes 3})$ under the choice of a reasonable section $U^{\otimes 3}\to \mathrm{Sym}^3U$. There is a canonical such map, given on sections by 
\[
s_1\otimes s_2\otimes s_3\mapsto \sum_{\sigma\in S_3}s_{\sigma(1)}\otimes s_{\sigma(2)}\otimes s_{\sigma(3)}
\]
under the hypothesis that $\mathrm{char}(k)\neq 2,3$.

\bibliographystyle{amsalpha}
\bibliography{borel}

\end{document}